\RequirePackage{fix-cm}
\documentclass[smallextended]{svjour3hack}       
\smartqed 
\usepackage{graphicx}
\usepackage{algorithm}
\usepackage{algorithmic}
\usepackage{amsfonts}
\usepackage{amsmath}
\usepackage{amssymb}
\usepackage{multirow}
\usepackage{verbatim}
\usepackage{subcaption}
\usepackage{longtable}
\usepackage{hyperref}
\usepackage{url}
\usepackage{ulem}
\usepackage{enumitem}
\setlist{itemsep=.05cm}

\newcommand{\var}[1]{\boldsymbol{#1}}
\newcommand*{\boldgreek}[1]{%
  \textpdfrender{%
    TextRenderingMode=FillStroke,%
    LineWidth=.4pt,%
  }{#1}%
}
\newcommand*{\boldDelta}[1]{%
  \textpdfrender{%
    TextRenderingMode=FillStroke,%
    LineWidth=.6pt,%
  }{#1}%
}

\numberwithin{theorem}{section}

\usepackage{amsopn}

\usepackage[T1]{fontenc}
\usepackage[utf8]{inputenc}

\usepackage{bm} 
\usepackage{amssymb}
\usepackage{amsmath}
\usepackage{amsfonts}
\usepackage{mathtools}
\usepackage{multirow}
\usepackage{verbatim}
\usepackage{longtable}
\usepackage{booktabs}

\usepackage{newtxtext,newtxmath}
\usepackage{pdfrender}

\usepackage{amssymb}

\begin{document}

\title{Improved Mixing and Pressure Loss Formulations for Gas Network Optimization}

\author{Geonhee Kim          \and
        Christopher Lourenco \and
        Daphne Skipper       \and
        Luze Xu
}

\authorrunning{Kim et.\ al} 

\institute{G. Kim \at
              Department of Mathematics, US Naval Academy \\
              \email{aabb2126@gmail.com}          
           \and
           C. Lourenco \at
            Department of Mathematics, US Naval Academy \\
            \email{lourenco@usna.edu}
            \and
            D. Skipper \at
            Department of Mathematics, US Naval Academy \\
            \email{skipper@usna.edu}
            \and
            L. Xu \at 
            Department of Mathematics, UC Davis \\
            \email{lzxu@ucdavis.edu}
}

\date{}

\maketitle

\begin{abstract}
Non-convex, nonlinear gas network optimization models are used to determine the feasibility of flows on existing networks given constraints on network flows, gas mixing, and pressure loss along pipes. 
This work improves two existing gas network models: a discrete mixed-integer nonlinear program (MINLP) that uses binary variables to model positive and negative flows, and a continuous nonlinear program (NLP) that implements complementarity constraints with continuous variables. 
We introduce cuts to expedite the MINLP and we formulate two new pressure loss models that leverage the flow-splitting variables: one that is highly accurate and another that is simpler but less accurate. 
In computational tests using the global solver BARON our cuts and accurate pressure loss improves: (1) the average run time of the MINLP by a factor of 35, (2) the stability of the MINLP by solving every tested instance within 2.5 minutes (the baseline model timed out on 25\% of instances), (3) the stability of the NLP by solving more instances than the baseline. Our simpler pressure loss model further improved run times in the MINLP (by a factor of 48 versus the baseline MINLP), but was unstable in the context of the NLP.

\end{abstract}

\newpage
\section{Introduction}

Mathematical programming is widely used to plan large scale gas distribution networks (e.g., \cite{fugenschuh2015chapter,geissler2015solving,hante2019complementarity,pfetsch2015validation,hamedi2011optimization,martin2006mixed}). One of the primary gas network optimization problems is the ``validation of nominations'': testing whether a given set of supplies and demands is feasible on an existing network (see \cite{geissler2015chapter} and \cite{rios2015reviewoptimization}, for example). Gas networks contain many engineering components that make their models extremely complex (nonlinear and nonsmooth), and therefore challenging in the context of optimization.

A new focus in gas network optimization is the mixing of gas from different sources. The nomination validation problem requires ensuring that the quality of gas at exit nodes meets the required heat power levels. Gas-mixing requires a so-called pooling model: an NP-hard class of non-convex mathematical optimization problems that are used to model mixing in the chemical industry (see \cite{greenberg1995poolinganalyzing}). Here we consider nomination validation models that incorporate gas mixing, in addition to nonlinear pressure-loss models to ensure that gas pressure at all nodes remains within normal operating ranges. 

\subsection{Previous Work}

Our work builds on \cite{hante2019complementarity}, which presents and tests alternative nonlinear gas network nomination validation optimization models that capture gas pressure and gas mixing. The gas pressure model accounts for pressure-loss due to friction in pipes, as well as compressors and valves, which increase and decrease pressure, respectively, as needed to remain within a given operating range. In order to model the mixing of gasses at network nodes, the flow along each arc is split into forward and backward components. One of the models, the MINLP (mixed-integer nonlinear program), utilizes integer variables to model flow-splitting. However, the primary focus of \cite{hante2019complementarity} is to formulate and test alternative formulations of complementarity constraints that model flow-splitting using only continuous variables, in NLPs (nonlinear programs). All of these models include a common pressure loss approximation originally derived for water networks \cite{DrinkingWater2009_opt_eng}. 

Computational tests using state-of-the-art optimization solvers on random instances of a 582-node gas network demonstrated that the NLP models were far superior to the MINLP model. In fact, the tests were so one-sided that the MINLP did not appear to be a viable option. Two versions of the NLP were tested on two local NLP solvers, and the MINLP was tested on one global and one local MINLP solver. All runs were given a 1-hour time limit. The most successful NLP solver / model combination achieved locally optimal solutions to 66\% of all tested instances; each of those successful instances solved within 10 seconds. In contrast, the local MINLP solver obtained only three locally optimal solutions to the MINLP, while the global solver was not able to solve the MINLP even once.

A primary goal of this work is to determine if improvements to the MINLP formulation, combined with advances in global MINLP solver technology, could make MINLP formulations of these kinds of problems competitive with NLP formulations.

\subsection{Our Contribution}
We began our work by reproducing the MINLP and the best-performing NLP introduced by \cite{hante2019complementarity}, which we refer to as the ``baseline’’ models. We tested the baseline models on random instances of the 582-node network using BARON, a state-of-the-art commercial optimization solver that seeks globally optimal solutions to MINLP and NLP models. These preliminary tests were consistent with the findings of \cite{hante2019complementarity}, although not quite as stark: the number of solved instances were similar for both models, but BARON obtained solutions to the NLP much faster than to the MINLP (26 versus 751 seconds, on average).

It is well-known that reducing the size of the integer-relaxed feasible region by including valid, formulation-strengthening constraints, or ``cuts'', can significantly improve global solver performance on mixed-integer programs (linear or nonlinear). The reason for this is that tighter formulations generate better bounds, leading to faster convergence of branch-and-bound algorithms. Thus, we propose and test three classes of integer-programming cuts aimed at improving the computational performance of the MINLP by tightening the gas-mixing formulation.

We also propose and test two pressure-loss models as simpler alternatives to the current state-of-the-art square root pressure-loss model \cite{DrinkingWater2009_opt_eng,hante2019complementarity}. Both of these simplified pressure-loss models make use of the flow-splitting variables required by the gas-mixing model, somewhat offsetting the performance impairments introduced by the inclusion of gas-mixing. The first model is tuned for accuracy with the best known pressure-loss model, the so-called HP-PC model, for both laminar and turbulent flows. The second, simpler model is a smooth version of an existing pressure loss approximation, the so-called PKr model, which is designed to be as accurate as possible in the high pressure case (which is consistent with normal operating ranges of natural gas networks) while using a simple, signed quadratic function of flow. Unlike the integer-programming cuts, the proprosed pressure-loss models apply to both the MINLP and the NLP because both models include flow-splitting variables.

In our computational experiments with BARON on random instances of the 582-node network, the inclusion of the integer-programming cuts enabled BARON to solve every instance of the MINLP within the 1 hour time limit, and reduced the average solve time by nearly 60\%.  For comparison, the baseline MINLP timed out on 24 out of 100 instances. This first (more accurate) pressure loss model improved the average solve time of the MINLP with cuts by more than 93\%, and improved the reliability of the NLP, nearly halving the number of timed-out instances. The second pressure loss model further improved solve times for the MINLP (by an additional 26\%), but proved to be unstable in the context of the NLP model. 

Overall, our computational tests provide evidence that, at least when seeking globally optimal solutions, MINLP models can be competitive with NLP models for gas network nomination models that include pressure-loss and mixing. Local NLP solvers have also improved since the publication of \cite{hante2019complementarity}, so additional testing is needed to compare the performance of local solvers on the NLP models versus global solvers on the MINLP models.

\subsection{Organization}
In Section \ref{sec:Background}, we set the stage by introducing gas network optimization and defining terminology and notation. We also define the common model components that appear in all of the models that we test. In Section \ref{sec:mixing}, we present the baseline models: two mixing formulations that were proposed by \cite{hante2019complementarity}, one discrete (the MINLP) and one continuous (the NLP). In Section \ref{sec:cuts}, we present three classes of integer-programming cuts aimed at improving the performance of the MINLP. In Section \ref{sec:pressureloss}, we present the pressure loss formulation from \cite{DrinkingWater2009_opt_eng,hante2019complementarity}, along with our modified pressure loss formulations. The derivation of the pressure loss formulations requires a detailed account of existing pressure loss models; that background is included here, as well. In Section \ref{sec:compresults}, we describe our computational experiments and present their results. We end with closing remarks in Section \ref{sec:conclusion}.

\section{Introduction to gas network optimization} \label{sec:Background}

In this section, we provide a brief overview of gas network optimization. Section \ref{sec:Background:Fundamentals} introduces the fundamental properties of the studied gas network optimization models, as well as key assumptions. Section \ref{sec:Background:Notation} defines notation that is common to all model variants. Section \ref{sec:Background:Template} introduces the basic template for all of the model variants we tested, including all common constraints and the common objective function. 

\subsection{Gas network fundamentals} \label{sec:Background:Fundamentals}

Gas network optimization models are typically based on a graph $G = (V,A)$, where $V$ is the set of nodes and $A$ is the set of arcs. The set of nodes are further subdivided into input nodes, $V_+$, exit nodes, $V_-$, and internal nodes $V_0$ (note that $+$, $-$, and $0$ represent the type of flow input/output at each of these network elements). The arcs represent transition elements in the network: pipes, compressors (where pressure can be increased), and control valves (where pressure can be reduced). In order to emphasize the gas mixing and pressure loss model components, we do not model the behavior of other transitional elements such as valves, resistors, or coolers, and instead treat them as zero-length pipes. For a more detailed treatment of the individual components of the model, we refer the reader to \cite{fugenschuh2015chapter,koch2015evaluating}.

Recent work in gas network optimization models \cite{hante2019complementarity,schmidt2016high} typically simplify the model by making several key assumptions, and we make the same assumptions. In particular, we assume that the network is stationary and isothermal, meaning that the state of gas and its temperature remain constant throughout the network. Furthermore, we assume each pipe has a slope of zero and has flow only in one direction, and the direction of flow in each pipe is unspecified (before optimizing). Likewise, for pressure loss models we assume that all molar masses are equal. 

\subsection{Notation} \label{sec:Background:Notation}

In this section, we give a brief overview of the sets, decision variables, and parameters used throughout our models. As much as possible 
(i.e., when it does not create a conflict), we have chosen notation that is consistent with the referenced literature. This makes it easier for the interested reader to translate between manuscripts. To further increase readability, all of our decision variables are shown in bold, making the variables easy to identify in the models. 

As described at the beginning of this section, a gas network is modeled on a graph $G = (V,A)$. The set of nodes, $V$, is divided into input, $V_+$, exit, $V_-$, and internal, $V_0$, nodes. Similarly, the arcs, $A$, in the network are partitioned into pipes, $A_{pi}$, compressors, $A_{cm}$, control valves, $A_{cv}$, and zero-length pipes, $A_0$. All transitional network elements that do not fall into one of the other three categories (e.g., valves, resistors, coolers) are treated as zero-length pipes.

The model variables capture the rate of flow, pressure, and calorific values of gas throughout the network.
 The fundamental variables common to all of our models are listed in Table \ref{tab:variables}. Pressure ($\boldDelta{\Delta}_{u,v}$, $\var{p}_u$) and calorific value ($\var{H}_u$, $\var{H}_{u,v}$) variables are all nonnegative, but  flow ($\var{q}_{u,v}$) can be positive or negative; the sign of $\var{q}_{u,v}$ represents the direction of flow along the arc $(u,v)$. There are additional flow variables that split the flow into positive and negative components. These and other auxiliary variables are introduced as needed. Again, we note that all decision variables are shown in bold to help differentiate variables and parameters.

\begin{table}[ht!]
\caption{Decision variables}
\label{tab:variables}
\centering \footnotesize
\begin{tabular}{cll} \toprule
Notation  & Description                   & Sets      \\ \midrule
$\var{p}_u$            & Gas pressure               & $u \in V$ \\
$\var{H}_u$        & Calorific value at the node & $u \in V$ \\
$\var{q}_{u,v}$     & Gas flow    & $(u,v) \in A$ \\
$\var{H}_{u,v}$        & Calorific value on the edge & $(u,v) \in A$ \\ 
$\boldDelta{\Delta}_{u,v}$ & Absolute pressure change & $(u,v) \in A_{cm} \cup A_{cv}$ \\
\bottomrule
\end{tabular}
\end{table}

Parameters appearing directly in the optimization model template provided in Section \ref{sec:Background:Template} are listed in Tables \ref{tab:node_parameters} and \ref{tab:arc_parameters}. We introduce additional parameters to accurately model pressure loss in Section \ref{sec:pressureloss}, where the pressure loss models are presented.

\begin{table}[ht!]
\caption{Node parameters}
\label{tab:node_parameters}
\centering \footnotesize
\begin{tabular}{clcc} \toprule
Notation & Description        & Sets  & Units \\ \midrule
$q^{nom}_u$      & Supply ($\geq 0$) or demand ($\leq 0$) & $u \in V$ & $m^3*s^{-1}$  \\
$H_u^{sup}$  & Calorific value of supplied gas & $u \in V_+$ & $J*m^{-3}$  \\ 
$\underbar{$p$}_u$, $\overline{p}_u$ & Pressure bounds (on $\var{p}_u$) & $u \in V$ & $Pa$ \\ 
$\underbar{$H$}_{u}$, $\overline{H}_{u}$ & Calorific value bounds (on $\var{H}_u$) & $u \in V$ & $J*m^{-3}$ \\
$\underbar{$P$}_u$, $\overline{P}_u$ & Heat power bounds \cite{hante2019complementarity} &  $u \in V_{-}$ & W \\
$\underbar{$q$}_u$, $\overline{q}_u$ & Node flow bounds & $u \in V_-$ & $m^3*s^{-1}$  \\ 
\bottomrule
\end{tabular}
\end{table}

\begin{table}[ht!]
\caption{Arc parameters}
\label{tab:arc_parameters}
\centering \footnotesize
\begin{tabular}{clcc} \toprule 
Notation & Description   & Sets & Units \\ \midrule
$\underbar{$q$}_{u,v}$, $ \overline{q}_{u,v}$ & Flow bounds & $(u,v) \in A$ &  $m^3*s^{-1}$ \\
$\underbar{$H$}_{u,v}$, $\overline{H}_{u,v}$ & Calorific value bounds & $(u,v) \in A$ & $J*m^{-3}$ \\
$\underbar{$\Delta$}_{u,v}$, $\overline{\Delta}_{u,v}$ & Pressure change bounds & $(u,v) \in A_{cm} \cup A_{cv}$ & $Pa$ \\
\bottomrule
\end{tabular}
\end{table}

\subsection{Optimization model template}\label{sec:Background:Template}

The models we consider are nonlinear optimization models that aim to capture the effects of gas flow, gas mixing, and pressure change throughout the pipes and nodes in the network. Below, we present the basic template for all of the gas network optimization models we consider. The constraints defined in the template are the same in all of the models. The model components that are not defined in the template (i.e., \eqref{cons:GeneralFlowSplit}, \eqref{cons:GeneralMixing}, and \eqref{cons:GeneralPressureLoss} below) take different forms in the models and are defined in later sections. The basic model template is as follows:
\begin{align}
\text{Minimize} \quad &  z ~:=~ \sum_{(u,v) \in A_{cm}} \boldDelta{\Delta}_{u,v}, \nonumber \\
\text{subject to} \quad &  q^{nom}_u ~=~  \sum_{(u,v) \in A}\var{q}_{u,v} - \sum_{(v,u) \in A}\var{q}_{v,u}, \quad &&\forall~ u \in V; \label{cons:massbalance}\\[0.2em]
& (Flow\text{-}splitting) \label{cons:GeneralFlowSplit} \\[0.5em]
& (Mixing) \label{cons:GeneralMixing} \\[0.5em]
& \boldDelta{\Delta}_{u,v} ~=~ \var{p}_v - \var{p}_u, &&\forall~ (u,v) \in A_{cm}; \label{cons:deltacm} \\[0.5em]
& \boldDelta{\Delta}_{u,v} ~=~ \var{p}_u - \var{p}_v, &&\forall~ (u,v) \in A_{cv}; \label{cons:deltacv} \\[0.5em]
& \var{p}_v^2 ~=~ \var{p}_u^2 - \boldgreek{\phi}_{u,v}, &&\forall~ (u,v) \in A_{pi}; \label{cons:pressureloss} \\[0.5em]
& (Pressure~loss) \label{cons:GeneralPressureLoss} \\[0.5em]
& \underbar{$P$}_u ~\leq~ \var{H}_u \vert q^{nom}_u \vert ~\leq~ \overline{P}_u, &&\forall~ u \in V_-; \label{cons:exitheatbounds} \\[0.5em]
& \underbar{$q$}_{u,v} ~\leq~ \var{q}_{u,v} ~\leq~ \overline{q}_{u,v},  &&\forall~ (u,v) \in A; \label{cons:qbound} \\[0.5em]
& \underbar{$p$}_u ~\leq~ \var{p}_u ~\leq~ \overline{p}_u, &&\forall~u \in V; \label{cons:pbound} \\[0.5em]
& \underbar{$H$}_{u,v} ~\leq~ \var{H}_{u,v} ~\leq~ \overline{H}_{u,v}, &&\forall~ (u,v) \in A; \label{cons:Hcabound} \\[0.5em]
& \underbar{$H$}_{u} ~\leq~ \var{H}_u ~\leq~ \overline{H}_{u}, &&\forall~ u \in V; \label{cons:Hcubound} \\[0.5em]
& 0 ~\leq~ \boldDelta{\Delta}_{u,v} ~\leq~ \overline{\Delta}_{u,v}, &&\forall~ (u,v) \in A_{cm} \cup A_{cv}. \label{cons:deltabound} 
\end{align}

The objective function aims to minimize the pressure increase that must be performed by compressors throughout the model. Constraints \eqref{cons:massbalance} enforce the balance of flow throughout the network. These flows are further used in flow-splitting \eqref{cons:GeneralFlowSplit} and mixing \eqref{cons:GeneralMixing}, which we describe in Section \ref{sec:mixing}. Next, \eqref{cons:deltacm} and \eqref{cons:deltacv} model the change in pressure across each compressor and control valve. Constraint \eqref{cons:pressureloss} models pressure loss in pipes; specifically $\boldgreek{\phi}_{u,v}$ models pressure loss along pipe $(u,v)$. The most accurate model for $\boldgreek{\phi}_{u,v}$ is nonsmooth and thus must be replaced by a smooth approximation. We describe a previously-implemented smooth approximation of $\boldgreek{\phi}_{u,v}$ and propose two new ones in Section \ref{sec:pressureloss}, as various options for \eqref{cons:GeneralPressureLoss}. Constraints \eqref{cons:exitheatbounds} give bounds on the heat power for all exit nodes. Finally, constraints \eqref{cons:qbound}-\eqref{cons:deltabound} are variable bounds.

Each of the missing model components has discrete and continuous versions, which we introduce in Sections \ref{sec:mixing} (flow-splitting and mixing) and \ref{sec:pressureloss} (pressure loss). We also introduce cuts, additional constraints aimed at tightening the model formulations, in Section \ref{sec:cuts}.

\section{Gas mixing model} \label{sec:mixing}

Gases with different calorific values mix at nodes of the network. The gas mixing constraints model the mixing of these gases at nodes, as well as the propagation of the newly mixed gases along arcs. Both formulations require the flow to be split into variables to capture the forward and backwards flow along each arc, so we begin by presenting the discrete and continuous flow-splitting formulations.

\subsection{Flow-Splitting} \label{sec:FlowSplitting}

The pipes in a gas network can have flow in either direction; mathematically this is modeled by allowing the flow variable, $\var{q}_{u,v}$, to be positive (for ``forward'' flow from $u$ to $v$) or negative (for ``backward'' flow from $v$ to $u$). We introduce two new variables $\boldgreek{\beta}_{u,v}$ and $\boldgreek{\gamma}_{u,v}$, which represent the positive and negative flow along $(u,v)$, respectively. 
These variables are required by the mixing constraints to isolate all of the flow into a node: the positive flow on arcs directed into the node plus the negative flow on arcs directed out of the node.

In this section, we describe how flow-splitting is modeled in both the discrete and continuous models. The reader is referred to \cite{hante2019complementarity}, which introduced and tested these models in the context of mixing in gas networks, for a more in-depth presentation of the flow-splitting models.

\subsubsection{Discrete flow-splitting} \label{sec:FlowSplitting:Discrete}

To model the positive and negative flows throughout the network, the discrete models utilize binary variables, $\var{d}_{u,v}$, defined as follows:
\[
\var{d}_{u,v} ~:=~ \begin{cases}
    1 & \text{if $\var{q}_{u,v} \geq 0$} \\
    0 & \text{if $\var{q}_{u,v} < 0$}
\end{cases},\quad ~\forall ~(u,v) \in A .
\]
Thus, if $d_{u,v} = 1$ the flow along arc $(u,v)$ is at least $0$, meaning that $\boldgreek{\beta}_{u,v} = \var{q}_{u,v}$ and $\boldgreek{\gamma}_{u,v} = 0$. Conversely, if $d_{u,v} = 0$ the flow along arc $(u,v)$ is negative, meaning that $\boldgreek{\beta}_{u,v} = 0$ and $\boldgreek{\gamma}_{u,v} = |\var{q}_{u,v}|$. We can model this splitting of flows linearly as follows, for all $(u,v) \in A$:
\begin{subequations}\label{cons:flowsplit_linear}
\begin{align}
    &\var{q}_{u,v} ~=~ \boldgreek{\beta}_{u,v} -\boldgreek{\gamma}_{u,v}; \\
    &0 ~\le~ \boldgreek{\beta}_{u,v} ~\le~ \var{d}_{u,v} \overline{q}_{u,v}; \\
    &0 ~\le~ \boldgreek{\gamma}_{u,v} ~\le~ (\var{d}_{u,v}-1)\underline{q}_{u,v}; \\
    &\var{d}_{u,v} ~\in~ \{0,1\}.
\end{align}
\end{subequations}

\subsubsection{Continuous flow-splitting}

The continuous models do not utilize the binary variables $\var{d}_{u,v}$. Positive and negative flows can be modeled as $\boldgreek{\beta}_{u,v}$ and $\boldgreek{\beta}_{u,v} - \var{q}_{u,v}$, respectively, via a simple nonlinear model with no additional variables (beyond $\boldgreek{\beta}_{u,v}$). However, the simplest version of such a model does not enforce mild complementarity \cite{luo1996mathematical,hante2019complementarity}, meaning that in the NLP case, standard constraint qualifications are violated at feasible points. To ensure complementarity, the following constraints are included for each $(u,v) \in A$:
\begin{subequations}\label{cons:Flow_Split_MPCC}
\begin{align}
    &\boldgreek{\beta}_{u,v} + \epsilon - \boldgreek{\mu}_{u,v,1} - \boldgreek{\mu}_{u,v,2} ~=~ 0; \label{cons:Flow_Split_MPCC_1}\\
    &\boldgreek{\beta}_{u,v} ~\ge~ 0;~~~~\boldgreek{\beta}_{u,v} - \var{q}_{u,v} ~\ge~ 0; \label{cons:Flow_Split_MPCC_2}\\
    &\boldgreek{\mu}_{u,v,1},~ \boldgreek{\mu}_{u,v,2} ~\ge~ 0; \label{cons:Flow_Split_MPCC_3}\\
    &\boldgreek{\mu}_{u,v,1}\boldgreek{\beta}_{u,v}~\leq~0;~~~~  \boldgreek{\mu}_{u,v,2}(\boldgreek{\beta}_{u,v} - \var{q}_{u,v})~\leq~0. \label{cons:Flow_Split_MPCC_4}
\end{align}
\end{subequations}
In these constraints, $\boldgreek{\beta}_{u,v}$ still represents the positive part of the flow along $(u,v)$, and the negative flow can be recovered as $\boldgreek{\beta}_{u,v} - \var{q}_{u,v}$. Two new auxiliary variables, $\boldgreek{\mu}_{u,v,1}$ and $\boldgreek{\mu}_{u,v,2}$, are defined on each arc and $\epsilon$ is a parameter (chosen to be $10^{-6}$ as in \cite{hante2019complementarity}). This formulation satisfies Abadie constraint qualification, which makes KKT conditions necessary for any local minimum of a sufficiently smooth cost function subject to these constraints \cite{hante2019complementarity}.

\subsection{Mixing and propogation}
Armed with the flow-splitting variables, we are ready to present the mixing and propagation models.

\subsubsection{Discrete mixing and propagation}

Using the positive and negative flow variables, we obtain a mixing formulation as follows:
\begin{subequations} \label{cons:mixing_betagamma_smooth} 
\begin{align}
    & \sum_{(v,u) \in A}\boldgreek{\beta}_{v,u}\var{H}_u + \sum_{(u,v) \in A}\boldgreek{\gamma}_{u,v}\var{H}_u \nonumber \\
&\quad\quad ~=~ \sum_{(v,u) \in A}\boldgreek{\beta}_{v,u} \var{H}_{v,u} + \sum_{(u,v) \in A}\boldgreek{\gamma}_{u,v} \var{H}_{u,v},&&\forall~u\in V\setminus V_+; \label{eqn:MINLP_mixing_1}\\
    &q^{nom}_u\var{H}_u + \sum_{(v,u) \in A}\boldgreek{\beta}_{v,u}\var{H}_u + \sum_{(u,v) \in A}\boldgreek{\gamma}_{u,v}\var{H}_u  \nonumber \\
&\quad\quad ~=~ \var{H}_u^{sup}q^{nom}_u + \sum_{(v,u) \in A}\boldgreek{\beta}_{v,u} \var{H}_{v,u} + \sum_{(u,v) \in A}\boldgreek{\gamma}_{u,v} \var{H}_{u,v},&&\forall~u\in V_+.\label{eqn:MINLP_mixing_2}
\end{align}
\end{subequations}
Likewise, the following constraints model propagation from nodes $u \in A$ to arcs directed away from \eqref{eqn:MINLP_propagation_1_linear} and towards \eqref{eqn:MINLP_propagation_2_linear} $u$:
\begin{subequations} \label{cons:propagation_betagamma} 
\begin{align}
    & -M_1(1-\var{d}_{u,v}) ~\le~ (\var{H}_u - \var{H}_{u,v}) ~\le~ M_2(1-\var{d}_{u,v}), &\forall~(u,v) \in A; \label{eqn:MINLP_propagation_1_linear}\\
    & -M_1 \var{d}_{v,u} ~\le~ (\var{H}_u - \var{H}_{v,u}) ~\le~ M_2 \var{d}_{v,u}, &\forall~(v,u) \in A,  \label{eqn:MINLP_propagation_2_linear}
\end{align}
\end{subequations}
where $M_1=\overline{H}_{u,v} - \underline{H}_{u}$, $M_2=\overline{H}_{u} - \underline{H}_{u,v}$, $M'_1=\overline{H}_{v,u} - \underline{H}_{u}$, and $M'_2=\overline{H}_{u} - \underline{H}_{v,u}$.

\subsubsection{Continuous mixing and propagation}

The continuous mixing formulation is identical to the discrete case except that the negative flow variable, $\boldgreek{\gamma}_{u,v}$, is replaced by $\boldgreek{\beta}_{u,v} - \var{q}_{u,v}$. Thus, we obtain the following continuous formulation:
\begin{subequations} \label{cons:mixing_MPCC} 
\begin{align}
    &\sum_{(v,u) \in A}\boldgreek{\beta}_{v,u}\var{H}_u + \sum_{(u,v) \in A} (\boldgreek{\beta}_{u,v} - \var{q}_{u,v})\var{H}_u \nonumber \\
&\quad\quad ~=~ \sum_{(v,u) \in A}\boldgreek{\beta}_{v,u} \var{H}_{v,u} + \sum_{(u,v) \in A} (\boldgreek{\beta}_{u,v} - \var{q}_{u,v}) \var{H}_{u,v}, && \forall~u\in V\setminus V_+;\label{eqn:MPCC_mixing1}\\
    &q^{nom}_u\var{H}_u + \sum_{(v,u) \in A}\boldgreek{\beta}_{v,u}\var{H}_u + \sum_{(u,v) \in A}(\boldgreek{\beta}_{u,v} - \var{q}_{u,v})\var{H}_u  \nonumber \\
&\quad\quad ~=~ \var{H}_u^{sup}q^{nom}_u + \sum_{(v,u) \in A}\boldgreek{\beta}_{v,u} \var{H}_{v,u} + \sum_{(u,v) \in A}(\boldgreek{\beta}_{u,v} - \var{q}_{u,v}) \var{H}_{u,v}; &&\forall~u\in V_+.\label{eqn:MPCC_mixing2}
\end{align}    
\end{subequations}

Without the binary variable $\var{d}_{u,v}$, the propagation constraints cannot be formulated linearly as in \eqref{cons:propagation_betagamma}. Instead, the propagation constraints of the NLP are modeled using bilinear terms as follows:
\begin{subequations} \label{cons:propagation_MPCC}
\begin{align}
    & (\var{H}_u - \var{H}_{u,v})\boldgreek{\beta}_{u,v} ~=~ 0, \quad &&\forall~ u \in V, (u,v) \in A; \label{eqn:MPCC_propagation_1_nl} \\
    & (\var{H}_u - \var{H}_{v,u})(\boldgreek{\beta}_{v,u} - \var{q}_{v,u}) ~=~ 0, \quad &&\forall~u \in V , (v,u) \in A. \label{eqn:MPCC_propagation_2_nl}
\end{align}    
\end{subequations}

\section{Cuts to accelerate mixing and flow formulations} \label{sec:cuts}

In its base state, the continuous optimization model outperforms the discrete model \cite{hante2019complementarity}; however, in general, discrete models can be greatly improved by adding cuts (additional constraints) to tighten their formulations, leading to improved computational performance. In this section, we introduce three types of cuts that significantly accelerate the discrete model. Specifically, Section \ref{sec:cuts:Mc} aims to improve the mixing constraints via the standard McCormick inequalities. Next, Sections \ref{sec:cuts:FlowDirection} and \ref{sec:cuts:dbounds} improve the discrete model by introducing cuts involving the binary variables, $\var{d}_{u,v}$. Specifically, Section \ref{sec:cuts:FlowDirection} introduces flow-direction cuts (similar to a set of cuts tested by \cite{Habeck2022} for models with potential-based flows) to accelerate flow-balance and flow-splitting in the discrete model, and Section \ref{sec:cuts:dbounds} introduces bounds using the binary variable $\var{d}_{u,v}$ on the bilinear terms in the mixing formulation.

\subsection{McCormick inequalities}  \label{sec:cuts:Mc}

Given a bilinear expression, $\var{x} \var{y}$, where the variables are bounded as $x_{\ell} \leq \var{x} \leq x_u$ and $y_{\ell} \leq \var{y} \leq y_u$, the McCormick inequalities \cite{mccormick1976computability} provide the convex hull formulation of the set $\{(\var{x},\var{y},\var{x}\var{y}): x_{\ell} \leq \var{x} \leq x_u \text{ and } y_{\ell} \leq \var{y} \leq y_u\}$. There are four such bilinear expressions in the mixing constraints \eqref{cons:mixing_betagamma_smooth}, $\boldgreek{\beta}_{v,u} \var{H}_u$, $\boldgreek{\gamma}_{u,v} \var{H}_{u}$, $\boldgreek{\beta}_{v,u} \var{H}_{v,u}$, and $\boldgreek{\gamma}_{u,v} \var{H}_{u,v}$. Below, \eqref{cons:McCormickMixing} apply the McCormick inequalities to these four expressions while removing redundant constraints (those expressions whose lower or upper bound is redundant with given variable bounds).
\begin{subequations} \label{cons:McCormickMixing}
\begin{align}
&\boldgreek{\beta}_{v,u}\var{H}_u ~\geq~ \overline{q}_{v,u}\var{H}_u + \boldgreek{\beta}_{v,u}\overline{H}_u - \overline{q}_{v,u}\overline{H}_u, \\
&\boldgreek{\beta}_{v,u}\var{H}_u ~\leq~ \overline{q}_{v,u}\var{H}_u + \boldgreek{\beta}_{v,u}\underbar{$H$}_u - \overline{q}_{v,u}\underbar{$H$}_u, \quad &&\forall~ u \in V,~(v,u) \in A; \\[.5em]
&\boldgreek{\gamma}_{u,v}\var{H}_u ~\geq~ |\underbar{$q$}_{u,v}|\var{H}_u + \boldgreek{\gamma}_{u,v}\overline{H}_u - |\underbar{$q$}_{u,v}|\overline{H}_u, \\
&\boldgreek{\gamma}_{u,v}\var{H}_u ~\leq~ |\underbar{$q$}_{u,v}|\var{H}_u + \boldgreek{\gamma}_{u,v}\underbar{$H$}_u - |\underbar{$q$}_{u,v}|\underbar{$H$}_u, \quad &&\forall~ u \in V,~(u,v) \in A; \\[.5em]
&\boldgreek{\beta}_{v,u}\var{H}_{v,u} ~\geq~ \overline{q}_{v,u}\var{H}_{v,u} + \boldgreek{\beta}_{v,u}\overline{H}_{v,u} - \overline{q}_{v,u}\overline{H}_{v,u}, \\
&\boldgreek{\beta}_{v,u}\var{H}_{v,u} ~\leq~ \overline{q}_{v,u}\var{H}_{v,u} + \boldgreek{\beta}_{v,u}\underbar{$H$}_{v,u} - \overline{q}_{v,u}\underbar{$H$}_{v,u}, \quad &&\forall~ u \in V,~(v,u) \in A; \\[.5em]
&\boldgreek{\gamma}_{u,v}\var{H}_{u,v} ~\geq~ |\underbar{$q$}_{u,v}|\var{H}_{u,v} + \boldgreek{\gamma}_{u,v}\overline{H}_{u,v} - |\underbar{$q$}_{u,v}|\overline{H}_{u,v}, \\
&\boldgreek{\gamma}_{u,v}\var{H}_{u,v} ~\leq~ |\underbar{$q$}_{u,v}|\var{H}_{u,v} + \boldgreek{\gamma}_{u,v}\underbar{$H$}_{u,v} - |\underbar{$q$}_{u,v}|\underbar{$H$}_{u,v}, \quad &&\forall~ u \in V,~(u,v) \in A.
\end{align}
\end{subequations}

Note that this variant of the McCormick inequalities is applied directly to the discrete model. For the continuous model, the constraints involving $\boldgreek{\gamma}_{u,v} \var{H}_u$ and $\boldgreek{\gamma}_{u,v} \var{H}_{u,v}$ are replaced by similar constraints involving $\var{q}_{u,v} \var{H}_u$ and $\var{q}_{u,v} \var{H}_{u,v}$. We experimented with other variants of the McCormick bounds (for example directly on $(\boldgreek{\beta}_{u,v}-\var{q}_{u,v}) \var{H}_u$) and bounds on $\var{q}_{u,v} \var{H}_u$ and $\var{q}_{u,v} \var{H}_{u,v}$ performed slightly better computationally.

\subsection{Flow-direction cuts} \label{sec:cuts:FlowDirection}

The nodes in the network, $V$, are partitioned into three subsets: internal nodes, $V_0$, entry nodes $V_+$, and exit nodes $V_-$. The next sets of constraints enforce flow rules on the nodes in the network:
\begin{subequations} \label{cons:flowcuts}
\begin{align} 
&\sum_{(u,v)\in A} \hspace{-.5em}\var{d}_{u,v} ~~+~ \sum_{(v,u)\in A} \hspace{-.5em}(1-\var{d}_{v,u}) ~\geq~ 1, &&\forall~u \in V_+ \cup V_0; \label{cons:entry_has_flowout} \\
&\sum_{(v,u)\in A} \hspace{-.5em}\var{d}_{v,u} ~~+~ \sum_{(u,v)\in A} \hspace{-.5em}(1-\var{d}_{u,v}) ~\geq~ 1, &&\forall~u \in V_- \cup V_0. \label{cons:exit_has_flowin}
\end{align}
\end{subequations}
Constraint \eqref{cons:entry_has_flowout} does not allow for all flows at a node to be in-flows. Constraint \eqref{cons:exit_has_flowin} enforces the same for out-flows.

\subsection{Bilinear Bounds for Mixing Constraints} \label{sec:cuts:dbounds}

The bilinear terms $\boldgreek{\beta}_{v,u} \var{H}_u$, $\boldgreek{\beta}_{v,u} \var{H}_{v,u}$, $\boldgreek{\gamma}_{u,v} \var{H}_u$, and $\boldgreek{\gamma}_{u,v} \var{H}_{u,v}$ appear in the gas mixing constraints \eqref{cons:mixing_betagamma_smooth}. Because $0 \leq \boldgreek{\beta}_{u,v} \leq \var{d}_{v,u} \overline{q}_{v,u}$, the first two bilinear terms must be zero when $\var{d}_{v,u} = 0$. The same is true for last two bilinear terms when $\var{d_{u,v}} = 1$. Thus, we explicitly provide bounds on these terms as follows:

\begin{subequations} \label{eqn:bilinearDbounds}
\begin{align}
\boldgreek{\beta}_{v,u} \var{H}_u &~\leq~ \var{d}_{v,u} \overline{q}_{v,u} \overline{H}_u,~~~~&\forall~u\in V,~(v,u) \in A; \\
\boldgreek{\beta}_{v,u} \var{H}_{v,u} &~\leq~ \var{d}_{v,u} \overline{q}_{v,u} \overline{H}_{v,u}, &\forall~u\in V,~(v,u) \in A; \\
\boldgreek{\gamma}_{u,v} \var{H}_u &~\leq~ (1-\var{d}_{u,v}) \vert \underbar{$q$}_{u,v} \vert \overline{H}_u, &\forall~u\in V,~(u,v) \in A; \\
\boldgreek{\gamma}_{u,v} \var{H}_{u,v} &~\leq~ (1-\var{d}_{u,v}) \vert \underbar{$q$}_{u,v} \vert \overline{H}_{u,v}, &\forall~u\in V,~(u,v) \in A.
\end{align}
\end{subequations}
\noindent The first two constraints enforce that if $\var{d}_{v,u} = 0$ then $\boldgreek{\beta}_{v,u} \var{H}_u$ and $\boldgreek{\beta}_{v,u} \var{H}_{v,u}$ are both zero. The third and fourth constraints enforce that if $\var{d}_{u,v} = 1$ then $\boldgreek{\gamma}_{u,v} \var{H}_u$ and $\boldgreek{\gamma}_{u,v} \var{H}_{u,v}$ are both zero. Note that the logic of these constraints is already present in the model; however, in our tests, including direct bounds on the bilinear terms slightly improved performance of the model.

\section{Pressure loss in pipes} \label{sec:pressureloss}

Pressure loss along pipes is modeled as $\boldgreek{\phi}_{u,v}$, which represents the square of the pressure lost due to friction along pipe $(u,v)$ in constraint \eqref{cons:pressureloss}. In general, this term is approximated via different functions, the most accurate of which is the so-called Hagen-Poisseuille Prandtl-Colebrook-White (HP-PC) pressure loss model (\S \ref{sec:pressureloss:HPPC}). This model is nonlinear, piecewise, and contains implicitly defined terms. The current state-of-the-art for optimization solvers, used for both the NLP and MINLP baseline models, is a square root approximation (\S \ref{sec:pressureloss:square}) which was originally developed for a water network \cite{DrinkingWater2009_opt_eng} and then adapted to gas networks \cite{fugenschuh2015chapter,schmidt2015high}. This square root model is smooth and twice differentiable; however, square root terms can be computationally expensive in the context of MINLP. Thus, we aim to expedite the computational performance of the pressure loss approximation by exploiting the flow splitting variables to derive two new pressure loss approximations that have better computational performance than the square root approximation. 

Throughout our discussion, we use the notation $\boldgreek{\phi}_{u,v}$ (with various superscripts) to represent pressure loss, rather than including variable(s) in the notation, e.g., $\boldgreek{\phi}_{u,v}(\var{q}_{u,v})$, with the understanding that all of the pressure loss models are functions of flow as represented by some combination of $\var{q}_{u,v}$, $\boldgreek{\beta}_{u,v}$, and $\boldgreek{\gamma}_{u,v}$.

Section \ref{sec:pressureloss:fundamentals} reviews the fundamental assumptions and parameters underlying these pressure loss models. Then, Sections \ref{sec:pressureloss:HPPC} and \ref{sec:pressureloss:square} present the highly accurate HP-PC and square root approximations. In Section \ref{sec:pressureloss:new}, we develop two new smooth pressure loss approximations that take advantage of the directional flow variables required for the mixing model.  In Section \ref{sec:pressureloss:compare},  we present graphs of all of the pressure loss formulations for both laminar and turbulent flows to compare the smoothing models to each other and to the baseline model.

\subsection{Fundamentals of Pressure Loss} \label{sec:pressureloss:fundamentals}

Loss of pressure as gas travels through a pipe $(u,v)$ is a function of the gas flow, $\var{q}_{u,v}$,  and the friction in the pipe. The friction in pipe $(u,v)$ is modeled by $\omega_{u,v} \lambda(\var{q}_{u,v})$ \cite{fugenschuh2015chapter},  where $\omega_{u,v}$ is a constant that captures the impact of several properties of the gas and the pipe and $\boldgreek{\lambda}(\var{q}_{u,v})$,  the so-called friction coefficient, is a function of flow. The constant, $\omega_{u,v}$, is defined by,
\begin{equation} \label{eq:FrictionFactor}
\omega_{u,v} := \frac{R_{s,m} z_m T_m L_{u,v}}{A_{u,v}^2 D_{u,v}},
\end{equation}
where the parameters are as described in Table \ref{tab:pressureloss_parameters}. The length, diameter, and area parameters are included in the model data for each pipe. For the other parameters, we follow the convention of \cite{fugenschuh2015chapter} and \cite{hante2019complementarity}: the ``$m$'' subscript indicates calculated mean values that are constant across all pipes in the model. We calculate these mean values as follows:
\begin{itemize}
    \item $R_{s,m}$ is a function of the specific gas constant and the molar masses of the gases in the network. Specifically, $R_{s,m}$ is calculated as the ratio of the universal gas constant, $R$, and the flow-weighted mean of the molar masses of the gases 
    entering the network at source nodes. We then assume that the gases have equal molar mass throughout the network so that the gas density is constant throughout the network (note that \cite{hante2019complementarity} makes the same assumption).
    
    \item $T_m$ is the flow-weighted mean temperature of all gas in the network. 
    \item $z_m$ is the compressibility factor and is the most complicated of these calculated parameters. We follow the lead of \cite{fugenschuh2015chapter} and apply the formula of Papay \cite{papay1968} to approximate the compressibility for a pipe $(u,v)$. We point the reader to \cite{fugenschuh2015chapter} for a full derivation of \eqref{eq:zm_definition}. In short, the compressibility is a function of the pressure and temperature of the gas in the pipes. We compute what is referred to as the ``reduced'' pressure and temperature for each pipe. We then utilize these to calculate flow-weighted means of pressure, $p_m$, pseudocritical pressure, $p_{c,m}$, temperature, $T_m$, and pseudocritical temperature, $T_{c,m}$. We then compute the global parameter $z_m$ as follows:
    \begin{equation} \label{eq:zm_definition}
    z_m ~=~ 1 - 3.52p_{c,m}e^{-2.26T_{c,m}}+0.247p_{c,m}^2e^{-1.878T_{c,m}}.
    \end{equation}
\end{itemize}

\begin{table}[ht!]
\caption{Pressure Loss Parameters}
\label{tab:pressureloss_parameters}
\centering \footnotesize
\begin{tabular}{clcc} \toprule
Symbol & Description   & Set/Value & Units \\ \midrule
$R_{s,m}$ & Mean specific gas constant &  & $J / (kg\cdot K)$\\
$T_m$ & Mean temperature &  & $K$  \\
$z_m$ & Mean compressibility factor &  & Unitless \\
$\omega_{u,v}$ & Friction constant & $(u,v) \in A_{pi}$ & $1/(m^2 s^2)$\\
$L_{u,v}$ & Length & $(u,v) \in A_{pi}$ & $m$ \\
$A_{u,v}$ & Cross-sectional area & $(u,v) \in A_{pi}$ & $m^2$ \\
$D_{u,v}$ & Diameter & $(u,v) \in A_{pi}$ & $m$ \\
$k_{u,v}$ & Roughness & $(u,v) \in A_{pi}$ & $m$ \\
$R$  & Universal gas constant & 8.3144621 & $J/(mol\cdot K)$ \\
$\eta$ & Dynamic viscosity & $10^{-6}$ & $kg/(m\cdot s)$\\
\bottomrule
\end{tabular}
\end{table}

Building on this foundation, we present several models for pressure loss, $\boldgreek{\phi}_{u,v}$, in Sections \ref{sec:pressureloss:HPPC}-\ref{sec:pressureloss:new}.

\subsection{Hagen-Poisseuille Prandtl-Colebrook–White (HP-PC)  pressure loss model}
\label{sec:pressureloss:HPPC}
The most accurate model for $\boldgreek{\lambda}(\var{q}_{u,v})$ is piecewise with different expressions for laminar and turbulent flows. Specifically, the model for $\boldgreek{\lambda}(\var{q}_{u,v})$ depends on the Reynolds number, $Re$, which is defined as $Re := \frac{D_{u,v}}{A_{u,v}\eta}|\var{q}_{u,v}|$, with $\eta$ representing the viscosity of the gas. If the Reynolds number is small ($\leq 2320$) the flow is laminar and $\boldgreek{\lambda}(\var{q}_{u,v})$ is modeled by the Hagen-Poisseuille equation \cite{finnemore2002fluid}. Conversely, if the Reynolds number is large, the flow is turbulent and is modeled by the Prandtl-Colebrook–White equation \cite{saleh2002fluid}. Altogether, we obtain the so-called HP-PC model for the friction coefficient:
\begin{equation}\label{eq:friction_HPPC}
 \begin{cases}
 \boldgreek{\lambda}(\var{q}_{u,v}) = \frac{64}{Re} 
 = \frac{64A_{u,v}\eta}{D_{u,v}|\var{q}_{u,v}|}, & \text{if }Re ~\leq~2320; \\
 \frac{1}{\sqrt{\boldgreek{\lambda}(\var{q}_{u,v})}} = -2\log\left(\frac{2.51}{Re\sqrt{\boldgreek{\lambda}(\var{q}_{u,v})}} + \frac{k_{u,v}}{3.71 D_{u,v}}\right), & \text{if }Re > 2320.
 \end{cases} 
 \end{equation}
Note that we follow the convention of \cite{hante2019complementarity} and set $\eta = 10^{-6}$ throughout our models. 

Now we have all the pieces to present the most accurate model for $\boldgreek{\phi}_{u,v}$, referred to as the HP-PC friction model \cite{DrinkingWater2009_opt_eng,schmidt2015high}:
\begin{equation} \label{eq:pressureloss:HPPC}
\boldgreek{\phi}^{HPPC}_{u,v} ~:=~ \omega_{u,v} \boldgreek{\lambda}(\var{q}_{u,v}) \vert \var{q}_{u,v} \vert \var{q}_{u,v},
\end{equation} where $\omega_{u,v}$ is the constant defined in \eqref{eq:FrictionFactor}, and the friction coefficient, $\boldgreek{\lambda}(\var{q}_{u,v})$, is defined as in \eqref{eq:friction_HPPC}. A function must be smooth to be suitable for use in most optimization solvers; that is, it must be continuous and first and second order differentiable. Thus, while $\boldgreek{\phi}^{HPPC}_{u,v}$ is the most accurate pressure loss model, it has characteristics that optimization solvers typically cannot handle: (1) $\boldgreek{\lambda}(\var{q}_{u,v})$ has a jump discontinuity, (2) $\boldgreek{\lambda}(\var{q}_{u,v})$ is defined implicitly for turbulent flows, and (3) the expression $\vert \var{q}_{u,v} \vert$ is nonsmooth at $\var{q}_{u,v}=0$ (recall that flow can take both positive and negative values). 

In the remainder of this section, we present and compare three smooth approximations of $\boldgreek{\phi}^{HPPC}_{u,v}$ for use in optimization models. Section \ref{sec:pressureloss:square} presents the current state-of-the-art, which uses a square root approximation to smooth the absolute value. In Section \ref{sec:pressureloss:new}, we derive two new, simpler, pressure loss models by leveraging the flow-splitting variables required by the mixing model. We provide a graphical comparison of pressure loss models in Section \ref{sec:pressureloss:compare}.

\subsection{Square root (sqrt) pressure loss model} \label{sec:pressureloss:square}

The current state-of-the-art is a smooth approximation of $\boldgreek{\phi}^{HPPC}_{u,v}$ that is derived in \cite{fugenschuh2015chapter,schmidt2015high} as an adaptation of a friction loss model developed for a water network in \cite{DrinkingWater2009_opt_eng}. The model includes additional parameters that improve accuracy relative to $\boldgreek{\phi}^{HPPC}_{u,v}$ for both laminar and turbulent flows. Specifically, this friction loss model smooths the HP-PC via a square root approximation as follows:
\begin{equation} \label{eq:pressureloss:squareroot}
\boldgreek{\phi}^{sqrt}_{u,v} ~:=~ \Lambda_{u,v} \left(\sqrt{\var{q}_{u,v}^2 + \hat{e}_{u,v}^2} + \hat{a}_{u,v} + \frac{\hat{b}_{u,v}}{\sqrt{\var{q}_{u,v}^2 + \hat{d}_{u,v}^2}}\right)\var{q}_{u,v},
\end{equation}
where 
\[{\Lambda}_{u,v} := \omega_{u,v}(2\log_{10}\rho_{u,v})^{-2},\]
and
\begin{equation*} \label{eq:pressureloss:squareroot:params1}
\hat{a}_{u,v}:=2t_{u,v}, ~~~\hat{b}_{u,v}:=(\ln \rho_{u,v} + 1)t_{u,v}^2 - \frac{\hat{e}_{u,v}^2}{2},
\end{equation*}
and
\begin{equation*} \label{eq:pressureloss:squareroot:params2}
t_{u,v} := \frac{2\alpha_{u,v}}{\rho_{u,v} \ln 10}, ~~~\alpha_{u,v} := \frac{2.51 A_{u,v} \eta}{D_{u,v}}, ~~~\rho_{u,v} := \frac{k_{u,v}}{3.71D_{u,v}},
\end{equation*}
where $\hat{d}_{u,v}$ and $\hat{e}_{u,v}$ are modeler-defined parameters. The parameters defined above are chosen so that $\boldgreek{\phi}^{sqrt}_{u,v}$ match the value and the first and second derivatives (all derivatives are taken with respect to flow, $\var{q}_{u,v}$) of $\boldgreek{\phi}^{HPPC}_{u,v}$ asymptotically (i.e., the turbulent case as $|\var{q}_{u,v}| \rightarrow \infty$), as well as the value and first and second derivative at $\var{q}_{u,v} = 0$ (improving accuracy in the laminar case). The reasoning for this choice of $\Lambda_{u,v}$ is explained in the next section in reference to \eqref{eq:friction_PKr}.

The modeler parameters, $\hat{d}_{u,v}$ and $\hat{e}_{u,v}$, must be defined appropriately to match the derivative of $\boldgreek{\phi}^{HPPC}_{u,v}$ at $\var{q}_{u,v} = 0$, and it is suggested that the relative contributions of the square root terms are balanced in \cite{DrinkingWater2009_opt_eng}.  Thus, in our implementation of this model, we set $({\boldgreek{\phi}^{sqrt}_{u,v}})'(0) = ({\boldgreek{\phi}^{HPPC}_{u,v}})'(0)$ with $\hat{d}_{u,v} = \hat{e}_{u,v}$.

$\boldgreek{\phi}^{HPPC}_{u,v}$ is linear in the laminar case, with
\begin{equation} \label{eq:laminar_derivative_at_0}
(\boldgreek{\phi}^{HPPC}_{u,v})'(0) ~=~ \frac{64\eta A_{u,v}\omega_{u,v}}{D_{u,v}}.
\end{equation}
Moreover, 
\[
(\boldgreek{\phi}^{sqrt}_{u,v})'(0) ~=~ \Lambda_{u,v}(\hat{e}_{u,v} + \hat{a}_{u,v} + \hat{b}_{u,v}/\hat{d}_{u,v})
\]
Thus, we take as $\hat{e}_{u,v}$ ($=\hat{d}_{u,v}$) the positive solution to the quadratic equation:
\begin{equation}\label{eq:modelerparam_quadratic}
\frac{1}{2}\hat{e}_{u,v}^2 
~+~ \left( \hat{a}_{u,v} - \frac{64 \eta A_{u,v}\omega_{u,v}}{D_{u,v}\Lambda_{u,v}}\right) \hat{e}_{u,v} 
~+~ (\ln \rho_{u,v} + 1)t_{u,v}^2
~=~ 0.
\end{equation}

\begin{theorem}
If $\frac{k_{u,v}}{D_{u,v}} < 1$, then \eqref{eq:modelerparam_quadratic}
has two real solutions: one positive and one negative.
\end{theorem}
\begin{proof}
The coefficient of the $\hat{e}_{u,v}^2$ term in \eqref{eq:modelerparam_quadratic} is positive, so the result follows if $(\ln \rho_{u,v} + 1)t_{u,v}^2 < 0$.

In the constant term, $t_{u,v}^2$ is strictly positive, so the term is negative if $\ln \rho_{u,v} < -1$. Thus, the constant is negative if $\rho_{u,v} < 1/e$, or equivalently, if $k_{u,v} / D_{u,v} < 3.71/e$, which follows because $k_{u,v}/D_{u,v} < 1$. 
\end{proof}

\begin{remark}
In practice, the ``relative roughness'' of a pipe, $k_{u,v} / D_{u,v}$, is significantly smaller than 1 for all modern pipe materials \cite{farshad2006surface}. Thus, for real world gas networks, \eqref{eq:modelerparam_quadratic} has two solutions, one positive and one negative.
\end{remark}

\subsection{New pressure loss models} \label{sec:pressureloss:new}

In this section, we introduce two new smoothing formulations of $\boldgreek{\phi}^{HPPC}_{u,v}$ aimed at improving the computational performance of both the integer and continuous optimization models. The first formulation is similar in spirit to $\boldgreek{\phi}^{sqrt}_{u,v}$ and achieves a similar level of accuracy, but achieves this with a quadratic expression. The second is equivalent to a common simplification of $\boldgreek{\phi}^{HPPC}_{u,v}$ that sacrifices accuracy, somewhat, for improved computational performance. 

\subsubsection{Flow-Splitting (fs) pressure loss model}
In this section, we derive a new pressure loss model that has a level of accuracy similar to that of $\boldgreek{\phi}^{sqrt}_{u,v}$, but with better algebraic properties for optimization applications.

From the derivation of $\boldgreek{\phi}^{sqrt}$ in \cite{DrinkingWater2009_opt_eng}, we see that $\boldgreek{\phi}^{sqrt}$ matches $\boldgreek{\phi}^{HPPC}$ (the most accurate pressure loss model) both asymptotically and at $\var{q}_{u,v}=0$ (in value and first and second derivatives). However, $\boldgreek{\phi}^{sqrt}$, while smooth, can still be computationally expensive. Recall the formulation,
\begin{equation} \tag{\ref{eq:pressureloss:squareroot}}
\boldgreek{\phi}^{sqrt}_{u,v} = \Lambda_{u,v} \left(\sqrt{\var{q}_{u,v}^2 + \hat{e}_{u,v}^2} + \hat{a}_{u,v} + \frac{\hat{b}_{u,v}}{\sqrt{\var{q}_{u,v}^2 + \hat{d}_{u,v}^2}}\right)\var{q}_{u,v},
\end{equation}
in which the square root expressions are essentially used to smooth $\vert \var{q}_{u,v} \vert$ in $\boldgreek{\phi}^{HPPC}$. In this section, we derive an alternative model, which we call the ``flow-splitting'' pressure loss model, that uses $\boldgreek{\beta}_{u,v}+\boldgreek{\gamma}_{u,v}$ as a smooth replacement for $\vert \var{q}_{u,v} \vert$. 

The derivation of our flow-splitting pressure loss model begins with an expression that is similar to $\boldgreek{\phi}^{sqrt}$, but that reintroduces $|\var{q}_{u,v}|$:
\[
\boldgreek{\phi}_{u,v}^{fs}~=~ \Lambda_{u,v} \left( |\var{q}_{u,v}| + \ddot{a}_{u,v} + \frac{\ddot{b}_{u,v}}{|\var{q}_{u,v}| + \ddot{d}_{u,v}}\right)\var{q}_{u,v},
\]
where $\ddot{a}_{u,v}$, $\ddot{b}_{u,v}$, and $\ddot{d}_{u,v}$ are parameters that allow for the model to be fine-tuned for laminar and turbulent flows, and $|\var{q}_{u,v}|$ is replaced by flow-splitting variables. In particular, $|\var{q}_{u,v}|$ is replaced by $\boldgreek{\beta}_{u,v}+\boldgreek{\gamma}_{u,v}$ in the discrete model and by $2\boldgreek{\beta}_{u,v}-\var{q}_{u,v}$ in the continuous model.

Note that $\boldgreek{\phi}^{\textit{~fs}}_{u,v}(0)=\boldgreek{\phi}^{HPPC}_{u,v}(0)=0$ as required. \autoref{thm:flowsplitting} provides the parameter values for which $\boldgreek{\phi}^{\textit{~fs}}_{u,v}$ is consistent with $\boldgreek{\phi}^{HPPC}_{u,v}$ asymptotically and at $\var{q}_{u,v}=0$.

\begin{theorem} \label{thm:flowsplitting}
If $\ddot{a}_{u,v} := \hat{a}_{u,v}$ and 
$\ddot{b}_{u,v} := \hat{b}_{u,v}+\hat{e}_{u,v}^2/2$,
then $\boldgreek{\phi}^{\textit{~fs}}_{u,v}$ agrees with $\boldgreek{\phi}^{HPPC}_{u,v}$ asymptotically in value and first and second derivative.
Moreover, if 
\begin{equation} \label{eq:ddot}
\ddot{d}_{u,v} ~=~ \left(\frac{D_{u,v}\Lambda_{u,v}}{64 \eta A_{u,v} \omega_{u,v}-2t_{u,v}D_{u,v}\Lambda_{u,v}}\right)\ddot{b}_{u,v},
\end{equation}
then
\[
(\boldgreek{\phi}^{\textit{~fs}}_{u,v})'(0) ~=~
(\boldgreek{\phi}^{HPPC}_{u,v})'(0).
\]
\end{theorem}

\begin{proof}
To streamline notation, we drop the ${u,v}$ subscripts and replace $\var{q}$ with $x$. $\boldgreek{\phi}^{\textit{~fs}}$ is an odd function of $x$, so we can analyze the expression for positive flows, thereby replacing $|x|$ with $x$. 

The asymptotic behavior of $\boldgreek{\phi}^{HPPC}$ is as follows (\cite{DrinkingWater2009_opt_eng}, page 54):
\begin{equation} \label{eq:HPPC_asymptotic_behavior}
\boldgreek{\phi}^{HPPC}(x) 
~ \sim ~
\Lambda(x^2 + 2tx + (\ln \rho + 1)t^2).
\end{equation}
Thus, the asymptotic error between $\boldgreek{\phi}^{\textit{~fs}}$ and $\boldgreek{\phi}^{HPPC}$ is,
\begin{align*}
E(x) &~=~ \Lambda(x^2 + \ddot{a}x + \ddot{b})
-  \Lambda(x^2 + 2tx + (\ln \rho + 1)t^2) \\
&~=~ \Lambda(\ddot{a} - 2t)x + \Lambda(\ddot{b} - (\ln \rho + 1)t^2).
\end{align*}
When $\ddot{a}$ and $\ddot{b}$ are as above, $\boldgreek{\phi}^{\textit{~fs}}$ agrees in value with $\boldgreek{\phi}^{HPPC}$ asymptotically. Similarly, with these values $\boldgreek{\phi}^{\textit{~fs}}$ and $\boldgreek{\phi}^{HPPC}$ agree asymptotically in first and second derivatives.

 We have 
\[
(\boldgreek{\phi}^{\textit{~fs}})'(x) ~=~ \Lambda \left( 2 x + \ddot{a} + \frac{\ddot{b}\ddot{d}}{(x + \ddot{d})^2}\right),
\]
so that, $(\boldgreek{\phi}^{\textit{~fs}})'(0) = \Lambda\left(\ddot{a}+\ddot{b}/\ddot{d}\right)$. From \eqref{eq:laminar_derivative_at_0}, $(\boldgreek{\phi}^{HPPC})'(0) ~=~ \frac{64\eta A\omega}{D}$.
Given that $\ddot{a} = 2t$, the value of $\ddot{d}$ in \eqref{eq:ddot} ensures agreement of the first derivatives at $x=0$.
\end{proof}

Note that the laminar part of $\boldgreek{\phi}^{HPPC}_{u,v}$ is linear in $\var{q}_{u,v}$, so that $(\boldgreek{\phi}^{HPPC}_{u,v})''(0)=0$. One drawback of $\boldgreek{\phi}^{\textit{~fs}}_{u,v}$ relative to $\boldgreek{\phi}^{sqrt}_{u,v}$ is that $(\boldgreek{\phi}^{sqrt}_{u,v})''(0)=0$, while $(\boldgreek{\phi}^{\textit{~fs}}_{u,v})''(0)$ is fixed at a nonzero value given the parameter values specified in \autoref{thm:flowsplitting}. In practice, we numerically evaluated $(\boldgreek{\phi}^{\textit{~fs}}_{u,v})''(0)$ on several pipes from real networks and observed that the resulting value was very small in magnitude; thus not significant.

Replacing $|\var{q}_{u,v} |$ as appropriate for each optimization model, we obtain the flow-splitting pressure loss expression for the discrete optimization model,
\begin{equation}\label{eq:PressureLoss:FlowSplitting:Discrete}
\boldgreek{\phi}^{\textit{~fs}^{(d)}}_{u,v} ~=~ \Lambda_{u,v}\left(\boldgreek{\beta}_{u,v}^2 - \boldgreek{\gamma}_{u,v}^2 + \ddot{a}_{u,v}(\boldgreek{\beta}_{u,v} - \boldgreek{\gamma}_{u,v}) + \frac{\ddot{b}_{u,v}(\boldgreek{\beta}_{u,v} - \boldgreek{\gamma}_{u,v})}{\boldgreek{\beta}_{u,v} + \boldgreek{\gamma}_{u,v} + \ddot{d}_{u,v}}\right),
\end{equation} 
and for the continuous optimization model,
\begin{equation} \label{eq:PressureLoss:FlowSplitting:Continuous}  
\boldgreek{\phi}^{\textit{~fs}^{(c)}}_{u,v} ~=~ \Lambda_{u,v}\left((2
\boldgreek{\beta}_{u,v} + \ddot{a}_{u,v}) \var{q}_{u,v}  - \var{q}_{u,v}^2 + \frac{\ddot{b}_{u,v} \var{q}_{u,v}}{2\boldgreek{\beta}_{u,v} - \var{q}_{u,v} + \ddot{d}_{u,v}}\right).
\end{equation}

\subsubsection{Smooth Prandlt-K\'{a}rm\'{a}n (PKr) model}

One common simplification of the HP-PC model is the law of Prandlt-K\'{a}rm\'{a}n for hydraulically rough pipes: 
\begin{equation}\label{eq:friction_PKr}
\lambda(\var{q}_{u,v}) \approx (2 \log_{10} \rho_{u,v})^{-2}, ~~~ \rho_{u,v} = \frac{k_{u,v}}{3.71 D_{u,v}}.
\end{equation}
This model, the so-called PKr model for the friction coefficient, is obtained by taking the limit of the turbulent ($Re > 2330$) portion of \eqref{eq:friction_HPPC} as $Re \rightarrow \infty$. By using \eqref{eq:friction_PKr}, rather than \eqref{eq:friction_HPPC}, in \eqref{eq:pressureloss:HPPC}, we arrive at the PKr model for pressure loss
$\boldgreek{\phi}^{PKr}_{u,v} ~:=~ \Lambda_{u,v} \vert \var{q}_{u,v} \vert \var{q}_{u,v}$,
where ${\Lambda}_{u,v} := \omega_{u,v}(2\log_{10}\rho_{u,v})^{-2}$ is a constant. This model is attractive because it is much simpler than the HP-PC pressure loss model and matches the quadratic portion of the HP-PC model for turbulent flows (see \eqref{eq:HPPC_asymptotic_behavior} for the asymptotic behavior of $\boldgreek{\phi}^{HPPC}_{u,v}$).

Recall from Section \ref{sec:FlowSplitting:Discrete} that the positive and negative portions of the flow are captured by the variables $\boldgreek{\beta}_{u,v}$ and $\boldgreek{\gamma}_{u,v}$, respectively,
so that $\var{q}_{u,v} = \boldgreek{\beta}_{u,v}-\boldgreek{\gamma}_{u,v}$
and $\vert \var{q}_{u,v} \vert = \boldgreek{\beta}_{u,v}+\boldgreek{\gamma}_{u,v}$. Thus, we have $\vert \var{q}_{u,v} \vert \var{q}_{u,v} = (\boldgreek{\beta}_{u,v} - \boldgreek{\gamma}_{u,v})(\boldgreek{\beta}_{u,v} + \boldgreek{\gamma}_{u,v})$. Simplifying this expression, we obtain a smooth bivariate pressure loss model that is equivalent to $\boldgreek{\phi}^{PKr}_{u,v}$ and appropriate for use in the discrete model:
\begin{equation} \label{eq:pressureloss:bivariate}
\boldgreek{\phi}^{PKr^{(d)}}_{u,v} ~:=~ \Lambda_{u,v} (\boldgreek{\beta}_{u,v}^2 - \boldgreek{\gamma}_{u,v}^2).
\end{equation}
The $\boldgreek{\gamma}_{u,v}$ variables do not exist in the continuous model, but we can remove those variables by substituting $\boldgreek{\gamma}_{u,v} = \boldgreek{\beta}_{u,v} - \var{q}_{u,v}$. After simplifying, we obtain the following smooth PKr formulation for use in the continuous model:
\begin{equation} \label{eq:pressureloss:univariate}
\boldgreek{\phi}_{u,v}^{PKr^{(c)}} ~:=~ \Lambda_{u,v} (2 \boldgreek{\beta}_{u,v} \var{q}_{u,v} - \var{q}_{u,v}^2).
\end{equation}

\subsection{Comparing Pressure Loss Models} \label{sec:pressureloss:compare}

To conclude this section, we graphically compare each of the pressure loss approximations to the most accurate pressure loss model, $\boldgreek{\phi}^{HPPC}_{u,v}$, at small (\autoref{fig:582LowFlow}) and large (\autoref{fig:582HighFlow}) values of flow. In the implementation of the optimization models, the approximation is computed for each pipe in the network. For illustrative purposes, we chose one pipe from the GasLib-582 network for the graphical comparison. Specifically, this pipe has the following values: diameter $=0.5m$, roughness $=10^{-5}m$, and modeler parameters for $\boldgreek{\phi}^{sqrt}_{u,v}$ computed as $\hat{e}=\hat{d}=0.49794$. Note that discrete and continuous versions of each pressure loss approximation are the same, so we include only one line in our graphs per approximation. 

As expected, $\boldgreek{\phi}^{PKr}_{u,v}$ is the least accurate for laminar flows because it is not tuned for consistency with $\boldgreek{\phi}^{HPPC}_{u,v}$ at $\var{q}_{u,v}=0$. Conversely, both $\boldgreek{\phi}^{sqrt}_{u,v}$ and $\boldgreek{\phi}^{\textit{~fs}}_{u,v}$ are tuned for accuracy and match the first derivative of $\boldgreek{\phi}^{HPPC}_{u,v}$ at zero. That said, we noticed a trend that $\boldgreek{\phi}^{sqrt}_{u,v}$ is more consistent with $\boldgreek{\phi}^{HPPC}_{u,v}$ for very small flows; however, $\boldgreek{\phi}^{\textit{~fs}}_{u,v}$ converges more quickly to the turbulent behavior of $\boldgreek{\phi}^{HPPC}_{u,v}$. In any case, both $\boldgreek{\phi}^{sqrt}_{u,v}$ and $\boldgreek{\phi}^{\textit{~fs}}_{u,v}$ give a much more accurate approximation of the actual pressure loss than the simpler $\boldgreek{\phi}^{PKr}_{u,v}$. Once the flows are increased to higher values, as in \autoref{fig:582HighFlow}, all models converge to the same quadratic function, except that $\boldgreek{\phi}^{PKr}_{u,v}$ very slightly underestimates pressure loss. While we are demonstrating these trends for a single pipe, we saw the same behavior when we produced these graphs for several pipes, as expected based on the analysis in Section \ref{sec:pressureloss:new}.

\begin{figure}[htbp]
  \centering
  \begin{subfigure}[b]{0.49\linewidth}
    \includegraphics[width=\linewidth]{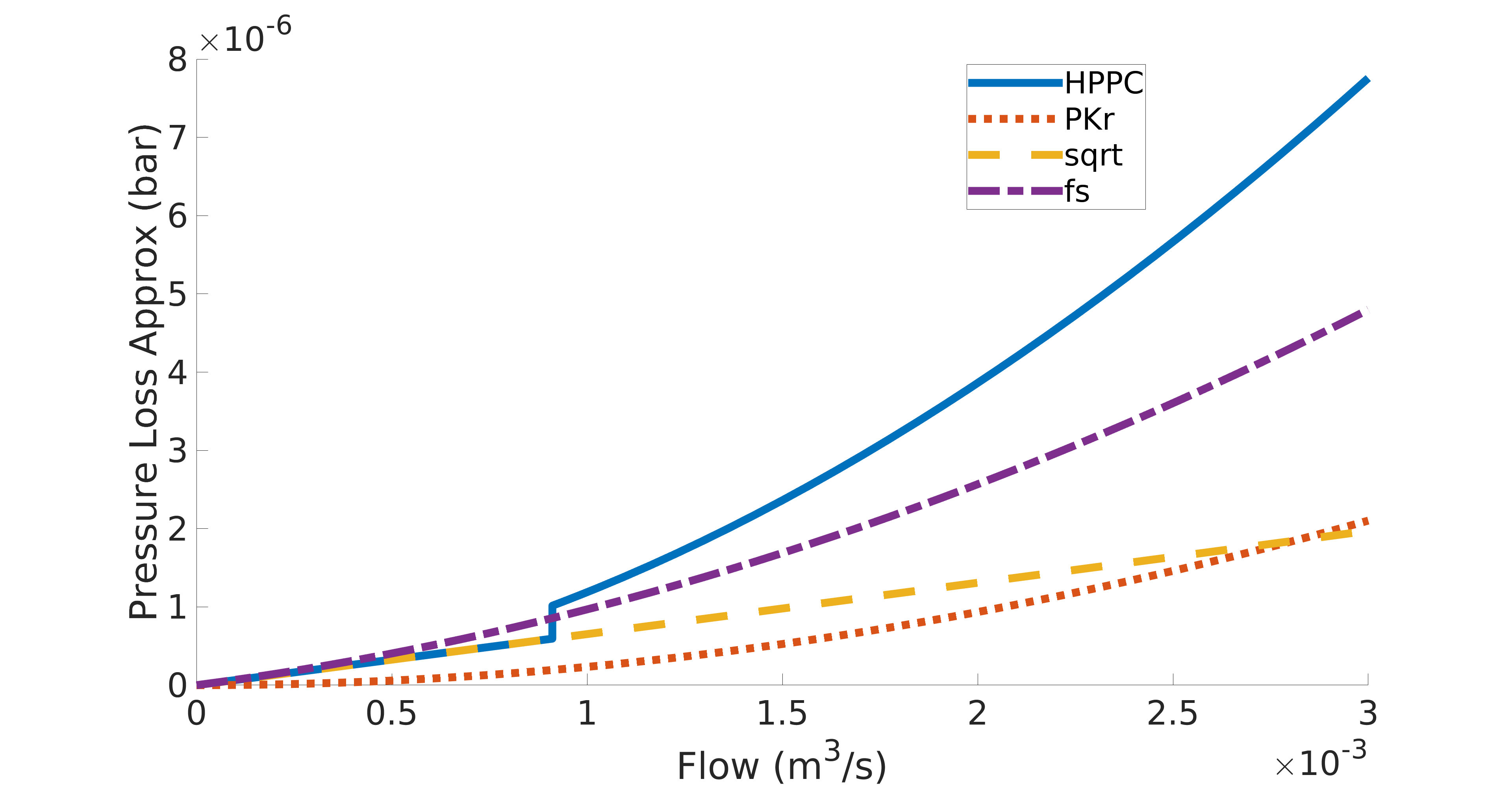}
    \caption{Low flow values}
    \label{fig:582LowFlow}
  \end{subfigure}
  \hfill
  \begin{subfigure}[b]{0.49\linewidth}
    \includegraphics[width=\linewidth]{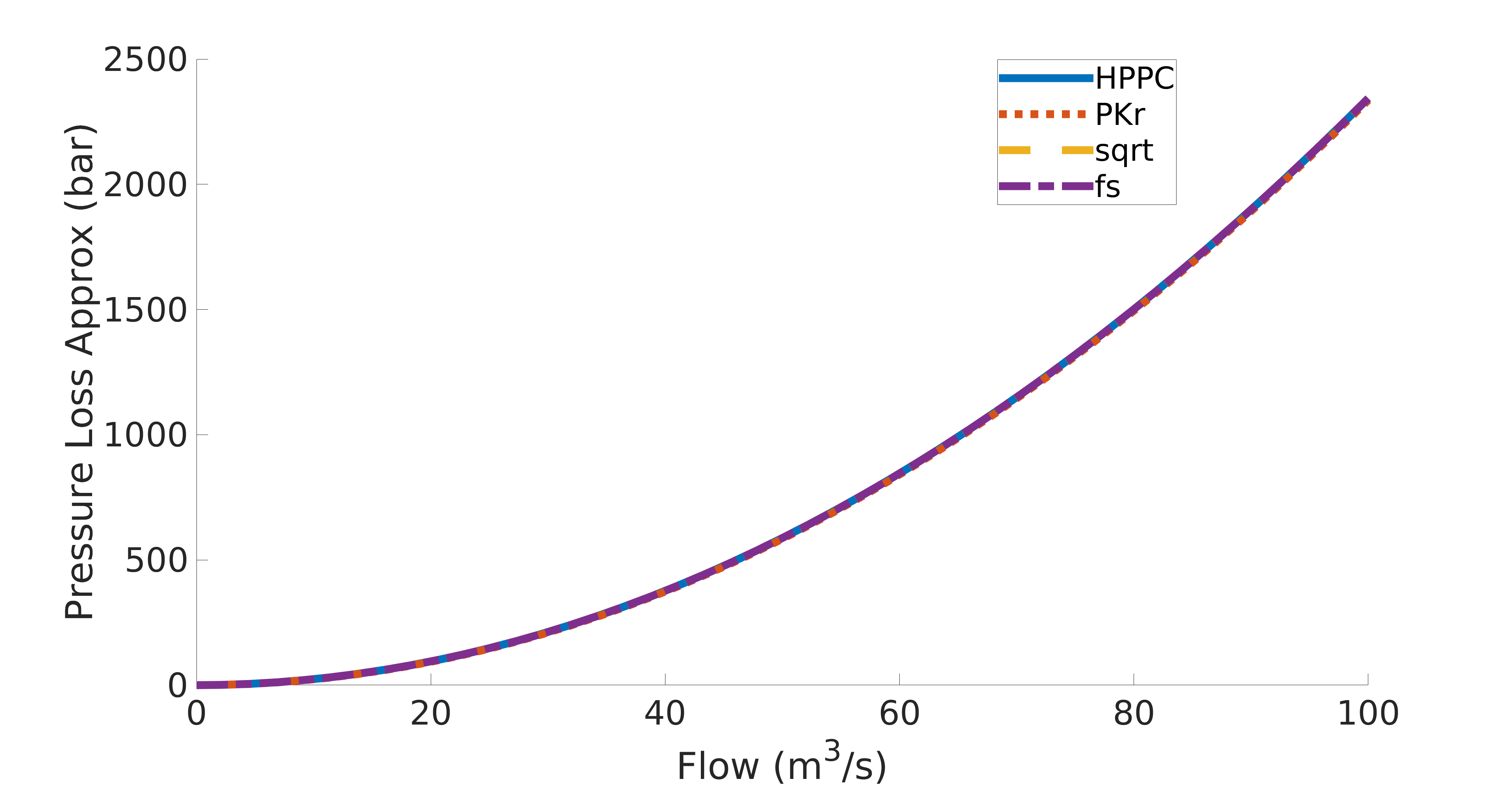}
    \caption{High flow values}
    \label{fig:582HighFlow}
  \end{subfigure}
  \caption{Pressure loss models for a sample GasLib582 Pipe}
  \label{fig:side-by-side}
\end{figure}

\section{Computational Results} \label{sec:compresults}

In this section, we computationally compare the effects of the various cuts and pressure loss formulations on the performance of both the discrete and continuous models. This section is organized as follows. Section \ref{sec:compresults:models} lists all of the models we tested. Section \ref{sec:compresults:environment} describes our computing environment. In Section \ref{sec:compresults:data}, we describe the source of our data (the GasLib repository), the instances we tested, and all of the conversions we performed to obtain the parameters we used in our models. Finally, we compare the run times of our models in Section \ref{sec:compresults:results}.

\subsection{Models Tested} \label{sec:compresults:models}

We consider two classes of models: discrete MINLPs and continuous NLPs. 

All MINLPs use the discrete flow-splitting \eqref{cons:flowsplit_linear}, mixing \eqref{cons:mixing_betagamma_smooth}, and propagation \eqref{cons:propagation_betagamma} formulations. We tested six variations of the MINLP: all three pressure loss models both without cuts and with all cuts. In particular, the discrete pressure loss models are $\boldgreek{\phi}^{sqrt}$ \eqref{eq:pressureloss:squareroot}, $\boldgreek{\phi}^{\textit{~fs}^{(d)}}$ \eqref{eq:PressureLoss:FlowSplitting:Discrete}, and $\boldgreek{\phi}^{PKr^{(d)}}$ \eqref{eq:pressureloss:bivariate}.
The cuts applied to the MINLP models are: the McCormick inequalities \eqref{cons:McCormickMixing}, the flow-direction cuts \eqref{cons:flowcuts}, and binary variable bounds on bilinaer terms \eqref{eqn:bilinearDbounds}.

All NLPs use the continuous flow-splitting \eqref{cons:Flow_Split_MPCC}, mixing \eqref{cons:mixing_MPCC}, and propagation \eqref{cons:propagation_MPCC} formulations.
We tested four variations of the NLP. 
We tested the three continuous versions of the pressure loss formulations:  $\boldgreek{\phi}^{sqrt}$ \eqref{eq:pressureloss:squareroot}, $\boldgreek{\phi}^{\textit{~fs}^{(c)}}$ \eqref{eq:PressureLoss:FlowSplitting:Continuous}, and $\boldgreek{\phi}^{PKr^{(c)}}$ \eqref{eq:pressureloss:univariate}.
The McCormick inequalities \eqref{cons:McCormickMixing} are the only cuts that apply to the continuous models. (The other categories of cuts involve the $\var{d}$ variables, which are only in the MINLP.) The results of the McCormick inequalities in the context of the NLP were so poor that we did not include all of the results; we included only the results from testing the McCormick inequalities in the NLP with $\boldgreek{\phi}^{sqrt}$ to demonstrate this phenomenon. 

\autoref{tab:AllModels} summarizes the models that we tested and provides the means, geometric means, and standard deviations of the solve times of each model applied to the same set of test instances from the GasLib test bank. The final column of the table indicates the number of instances that solved to global optimality.

\subsection{Computing Environment} \label{sec:compresults:environment}

We tested our models with the global nonlinear optimization solver, BARON \cite{sahinidis1996BARON} using default settings. We chose BARON for the following reasons: (1) it is a global solver, (2) it has been used in the gas network literature to solve mixing problems, (3) it can correctly handle integer variables, and (4) it is among the best-performing nonlinear global solvers. We generated our data files using Python, implemented our models in AMPL 20230817 \cite{fourer1987ampl}, and solved them using BARON version 23.3.11 via the AMPL command line interface. We recorded solve times using AMPL's \verb|_total_solve_time| metric. All computations were performed on a standard desktop computer running Ubuntu 18.04 with an Intel i9-10900K CPU and 32 GB of RAM. All of our data and model files are publicly available at \url{https://github.com/clouren/Gas-Optimization}.

We note that we also tried the global nonlinear solver SCIP \cite{bestuzheva2021scip} and its performance was not competitive with BARON. Similarly, it is valid to utilize a local solver for the NLP due to the complementarity constraints. We tried IPOPT \cite{wachter2006implementation} however, it was not competitive with BARON. Budgetary constraints and our primary focus on improving the MINLP formulation drove our decision to purchase only one commercial solver license, BARON, for our tests, so we were not able to test on a commercial local solver. We believe BARON is still a reasonable choice for the NLP for our tests because we are obtaining global solutions to the MINLP and, as shown in Section \ref{sec:compresults:results}, BARON's performance on the NLP agrees with the results of \cite{hante2019complementarity} utilizing commercial local solvers.

\subsection{Data and conversions} \label{sec:compresults:data}

We ran our computational tests on data from the GasLib repository, a library of realistic gas network and nomination (supply and demand) instances found at \url{https://gaslib.zib.de/} and described in \cite{GasLibInstances}. Specifically, the repository contains seven named networks (GasLib-11, GasLib-24, GasLib-40, GasLib-134, GasLib-135, GasLib-582, and GasLib-4197 where the number indicates the number of nodes in the network), each of which has one network file containing the structure of the nodes/arcs, and one scenario file containing ingress and egress flows, pressure values, etc. Furthermore, in addition to the original scenario file, GasLib-582 and GasLib-4197 contains several hundred so-called ``nomination'' files, which are additional scenario files operating on the same network. 

The instances with fewer than 582 nodes have very small run times and are intended mostly for developing and debugging models; thus, we do not consider them. We select 100 total instances from the GasLib-582 network: its original scenario file as well as 99 of the randomly selected ``nomination'' scenario files. Likewise, none of the nonlinear optimization techniques in the literature (i.e.,  \cite{hante2019complementarity,koch2015evaluating,schmidt2016high}) consider the large 4197 network due to its large size and the computational cost required to find feasible solutions. For this reason, we also do not consider this network.

The GasLib repository provides the vast majority of parameters directly, with the following caveats:
\begin{itemize}
    \item Heat power is defined for each $u \in V_-$ and is measured as $\var{H}_u \vert q_{u}^{nom} \vert $. We utilize similar heat power bounds to those of \cite{hante2019complementarity}, namely:
\[
\underbar{$P$}_u := (0.9) \overline{q}_u H_m 
\quad \text{ and } \quad 
\overline{P}_u := (1.1)\underbar{$q$}_u H_m, 
\quad \text{ for all } u \in V_-,
\]
Note that since these bounds are at sink nodes, the flow bounds are nonpositive; thus, counterintuitively, the term $\overline{q}_u$ is used for the lower bound and $\underbar{$q$}_u$ is for the upper bound. Likewise, $H_m$ represents the flow-weighted mean calorific value of all gas entering the network,
\[
H_m := \frac{\sum_{u \in V_+} q^{nom}_u H^{sup}_u}
{\sum_{u \in V_+} q^{nom}_u}.
\]
    \item We calculated global averages for: molar mass, gas temperature, pseudocritical temperature, and specific gas constant. These are used to calculate the compressibility factor described in Section \ref{sec:pressureloss:fundamentals}.
    \item Calorific values are provided for the nodes, not arcs, in the network. Thus, for each arc, we calculate a global calorific upper and lower bound as the maximum and minimum calorific values across all nodes, respectively.
\end{itemize}

Finally, we convert all units given into their SI standards. Specifically, flow is converted into $m^3/s$, molar mass is converted to $kg/mol$, all temperatures are converted to $K$, and all lengths are converted to $m$.

These standard units ($m^3/s$) to measure flow works in all except the pressure loss constraints. In particular, in $\var{p}_u^2 - \var{p}_v^2 = \boldgreek{\phi}_{u,v}$, pressure is measured in $Bar$. Thus, for the pressure loss models, we multiply the flow by its appropriate density in order to have units of $kg/s$. The pressure loss is then computed in terms of $Pa^2$. Then we convert from $Pa^2$ to $Bar^2$ by dividing by $(10^5)^2$.

\subsection{Results} \label{sec:compresults:results}

In this section, we present the results of our computational study of the ten models described in Section \ref{sec:compresults:models} when applied to the 100 instances described in Section \ref{sec:compresults:data}. All models were given a 1 hour limit. Table \ref{tab:AllModels} summarizes the ten models and lists the means, geometric means, and standard deviations of their solve times on the tested 100 GasLib-582 instances. It also lists the number of instances of each model that solved to optimality, reached the 1 hour time limit, and were mislabeled as infeasible. From these results, we make the following observations:

\begin{itemize}
\item BARON had some trouble with the MINLP before cuts were added, timing out on 24 instances with $\boldgreek{\phi}^{sqrt}$ and wrongly classifying 10 instances with $\boldgreek{\phi}^{\textit{~fs}}$ as infeasible. However, adding cuts stabilized and improved the performance of the MINLP model in the context of all three pressure loss formulations: all 100 instances solved to optimality, and did so much faster on average than the same model without cuts.

\item The NLP had the best performance with $\boldgreek{\phi}^{sqrt}$ without McCormick cuts, where summary solve times (for those instances reaching optimality) were competitive with the fastest versions of the MINLP. However, this model timed out on 19 instances. NLP with $\boldgreek{\phi}^{\textit{~fs}}$ had slightly slower solve times on solved instances, but timed out on only 11 instances. The NLP performed very poorly with the simple $\boldgreek{\phi}^{PKr}$ formulation; BARON incorrectly classified 48 instances as infeasible and timed out on 10 instances. 
\item The addition of cuts universally improved the performance of the MINLP but impaired the performance of the NLP. This general trend can be seen by the smaller geometric mean of the MINLP with cuts but higher geometric mean of the NLP with McCormick. These results are investigated in further detail in Section  \ref{sec:compresults:results:Cuts}.
\item Summary results indicate that MINLP with $\boldgreek{\phi}^{\textit{~fs}}$ and cuts is the best model for balancing computational performance and accuracy. This model had consistent performance across the instances, with average and geometric mean run times around 21.1 and 17.2 seconds, respectively, and with a relatively small standard deviation of 18.4 seconds. Furthermore, 97 out of 100 instances of this model solved within 1 minute.
\end{itemize}

\begin{table}[h!]
\centering \footnotesize
\begin{tabular}{lllrrrc}
\toprule
Model & $\boldgreek{\phi}^*$ & Cuts & Mean & Geo mean & St dev & Optimal / Limit / Infeasible \\
\midrule
MINLP$^{\dagger}$ & $sqrt$ & None  & 750.6 & 246.6 & 953.0 & 76 / 24 / 0 \\
MINLP & $sqrt$ & All & 309.2 & 129.3 & 460.8  & 100 / 0 / 0 \\
\midrule
MINLP & $fs^{(d)}$ & None & 105.3 & 68.6 & 150.4 & 90 / 0 / 10 \\
MINLP & $fs^{(d)}$ & All & 21.1 & 17.2 & 18.4 & 100 / 0 / 0 \\
\midrule
MINLP & $PKr^{(d)}$  & None  & 162.0 & 71.7 & 211.5 & 100 / 0 / 0 \\
MINLP & $PKr^{(d)}$  & All & 15.6 & 13.9 & 7.9 & 100 / 0 / 0 \\
\midrule
NLP$^{\dagger}$ & $sqrt$ & None & 26.3 & 9.7 & 77.5 & 81 / 19 / 0 \\
NLP & $sqrt$ & McC & 132.9 & 44.0 & 378.1 & 85 / 15 / 0 \\
\midrule
NLP & $fs^{(c)}$ & None & 39.8 & 14.6 & 80.2 & 89 / 11 / 0 \\
\midrule
NLP & $PKr^{(c)}$ & None & 378.9 & 80.5 & 606.2 & 42 / 10 / 48 \\
\bottomrule
\end{tabular}
\caption{Summary results for 100 gas network nomination instances. The mean, geometric mean, and standard deviation of solve times (in seconds) are restricted to instances that solved within a 1 hour time limit. The final column indicates the number of instances that solved to global optimality, timed out (limit), or were mislabeled as infeasible. $^\dagger$ These are the baseline NLP and MINLP models.}
\label{tab:AllModels}
\end{table}

In the ensuing sections, we investigate these findings in more detail using Dolan and Moré \cite{dolan2002benchmarking} performance profiles. Briefly, a performance profile is a graph used to compare algorithms (models in our case) taking into account the number of instances solved, as well as the cost required to solve each instance. The performance of each algorithm corresponds to a curve on the graph, where each point on the curve represents the percentage of instances (y-axis) the algorithm solved within a time-multiple (x-axis) of the fastest solve time (among all algorithms represented in the figure) for each instance. An important property of performance profiles is that they are insensitive to the relative difficulty among instances (i.e., they are not biased towards easy or hard instances). This is because, for any instance, all solution times are relative to the fastest solver on that instance. The simplest interpretation of the graph is that the highest curve corresponds to the best-performing algorithm. For the ensuing comparative analysis, if an instance timed out, its time was set to be 1 hour for comparison purposes.

\subsubsection{Baseline Models}

Hante and Schmidt \cite{hante2019complementarity} performed a pairwise comparison of the baseline MINLP and NLP models (with $\boldgreek{\phi}^{sqrt}$ and no cuts) and concluded that the NLP had better performance. Our results agree with this finding. NLP was able to solve more instances overall (81 as opposed to 76) and was faster for 80\% of solved instances. Furthermore, the baseline NLP completely dominates the baseline MINLP in summary statistics; specifically, the MINLP average and geometric mean solve times were 28.5 and 25.4 times slower, respectively, than those of the NLP. The NLP model is the clear winner in the performance profile in Figure \ref{fig:BaselineModels}.

\begin{figure}
    \centering
    \includegraphics[width=0.75\textwidth]{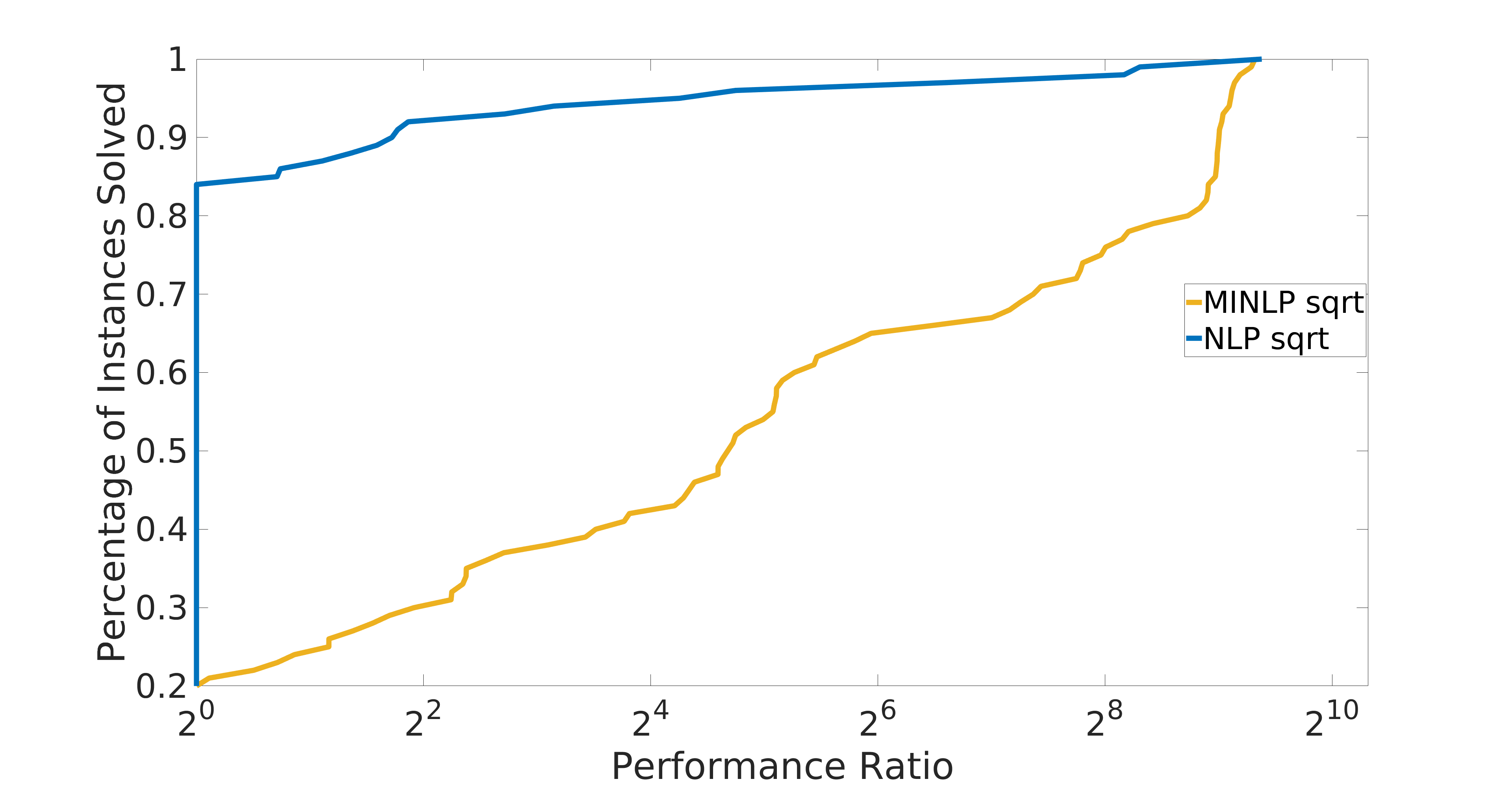}
    \caption{Baseline NLP model dominates baseline MINLP}
    \label{fig:BaselineModels}
\end{figure}

\subsubsection{Effects of Cuts} \label{sec:compresults:results:Cuts}

Next, we investigate the impact of adding cuts to the NLP and MINLP models. For the NLP, the only applicable cuts are the McCormick inequalities; these had a significantly negative impact on all variants of the NLP model in our tests. As an example to demonstrate the impact of cuts on the performance of the NLP models, we compare the performance of the baseline NLP with and without McCormick inequalities. The McCormick inequalities dramatically increased the NLP run times: they increased the average and geometric mean run times by factors of 5.1 and 4.5, respectively, as indicated in Table \ref{tab:AllModels}.

Conversely, the cuts significantly improved the performance of the MINLP with all pressure loss variants. Specifically, adding cuts to the MINLP improved $\boldgreek{\phi}^{sqrt}$, $\boldgreek{\phi}^{\textit{~fs}}$, and $\boldgreek{\phi}^{PKr}$ by factors of 2.4, 5.0, and 10.4, on average respectively, and 1.9, 4.0, 5.2, in geometric mean respectively.

The performance profiles in Figure \ref{fig:perf:CutsAllEffects} confirm all of these findings: the MINLP is significantly accelerated by the addition of cuts, while the McCormick inequalities hinder the performance of the NLP. Most notably, the cuts had extremely positive effects on the MINLP with $\boldgreek{\phi}^{\textit{~fs}}$ and $\boldgreek{\phi}^{\textit{PKr}}$, as with cuts the models were fastest for over 90\% of instances.

\begin{figure}[t!]
    \centering
    \begin{subfigure}[t]{0.49\textwidth}
        \centering
        \includegraphics[width=0.99\textwidth]{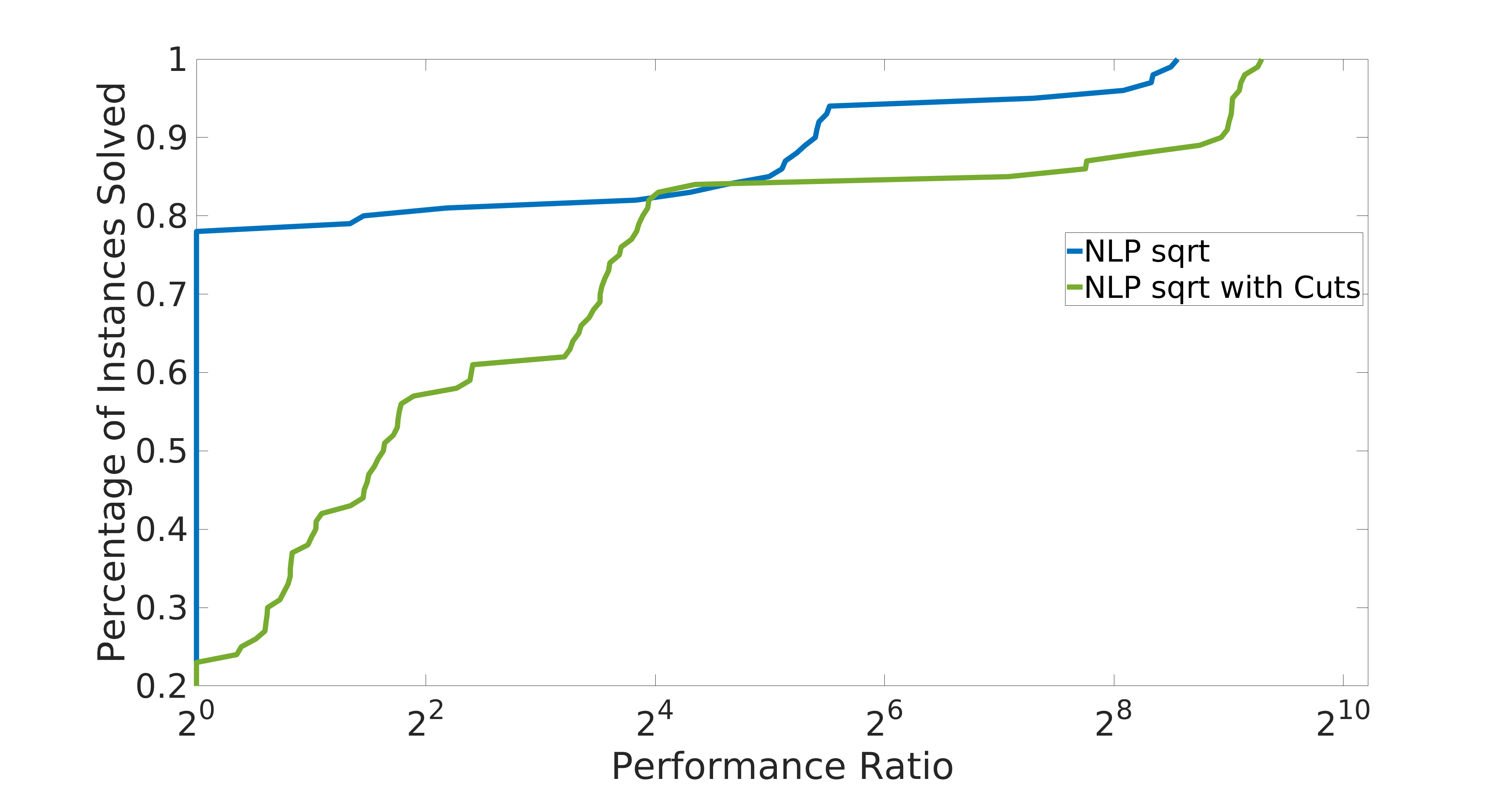}
        \caption{McCormick cuts significantly hamper NLP}
        \label{fig:perf:Cuts_NLP}
    \end{subfigure}%
    \begin{subfigure}[t]{0.49\textwidth}
        \centering
		\includegraphics[width=0.99\textwidth]{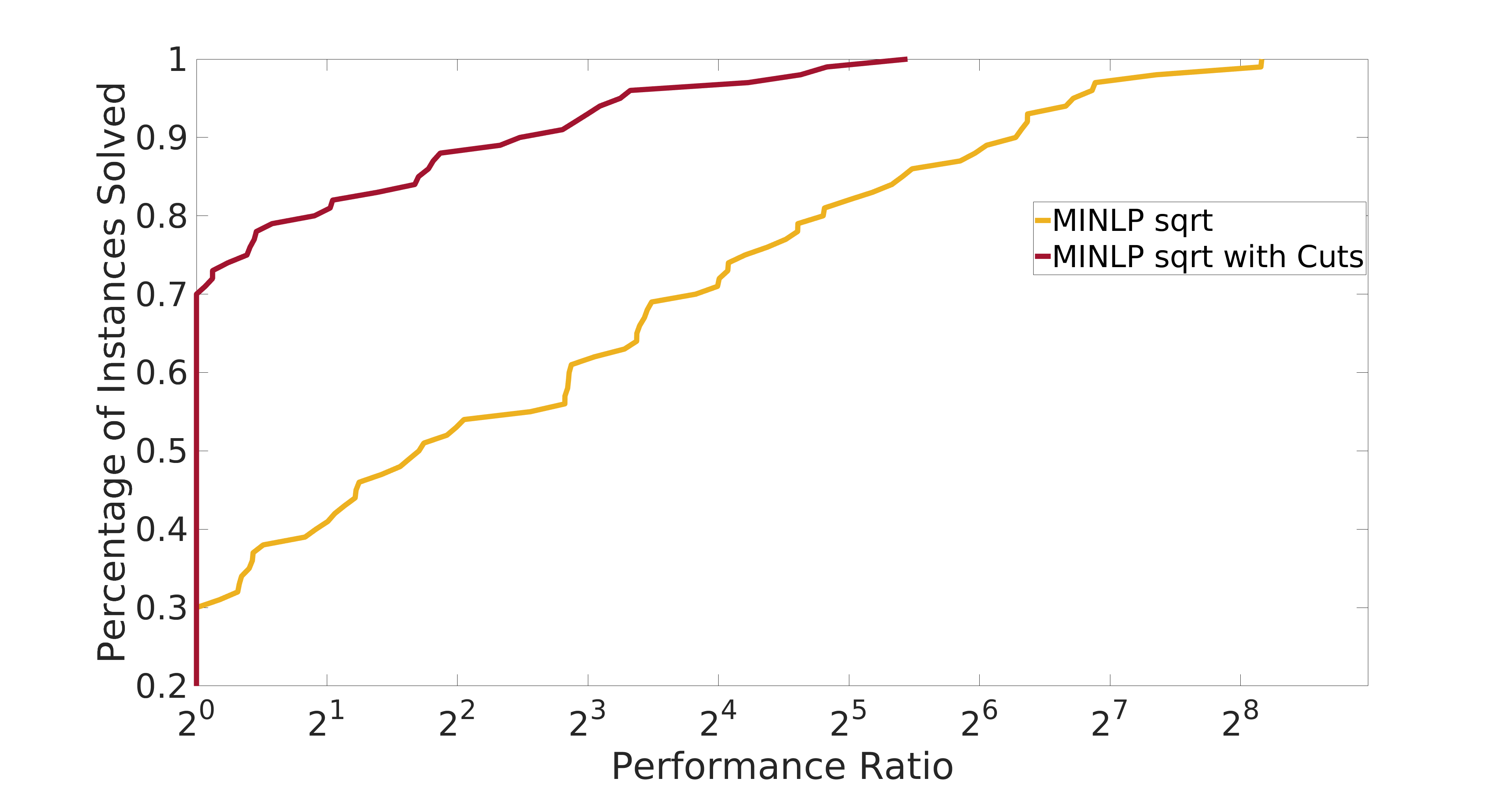}
        \caption{Cuts significantly accelerate MINLP with $\boldgreek{\phi}^{sqrt}$}
        \label{fig:perf:Cuts_MINLP_SR}
    \end{subfigure}

    \begin{subfigure}[t]{0.49\textwidth}
        \centering
		\includegraphics[width=0.99\textwidth]{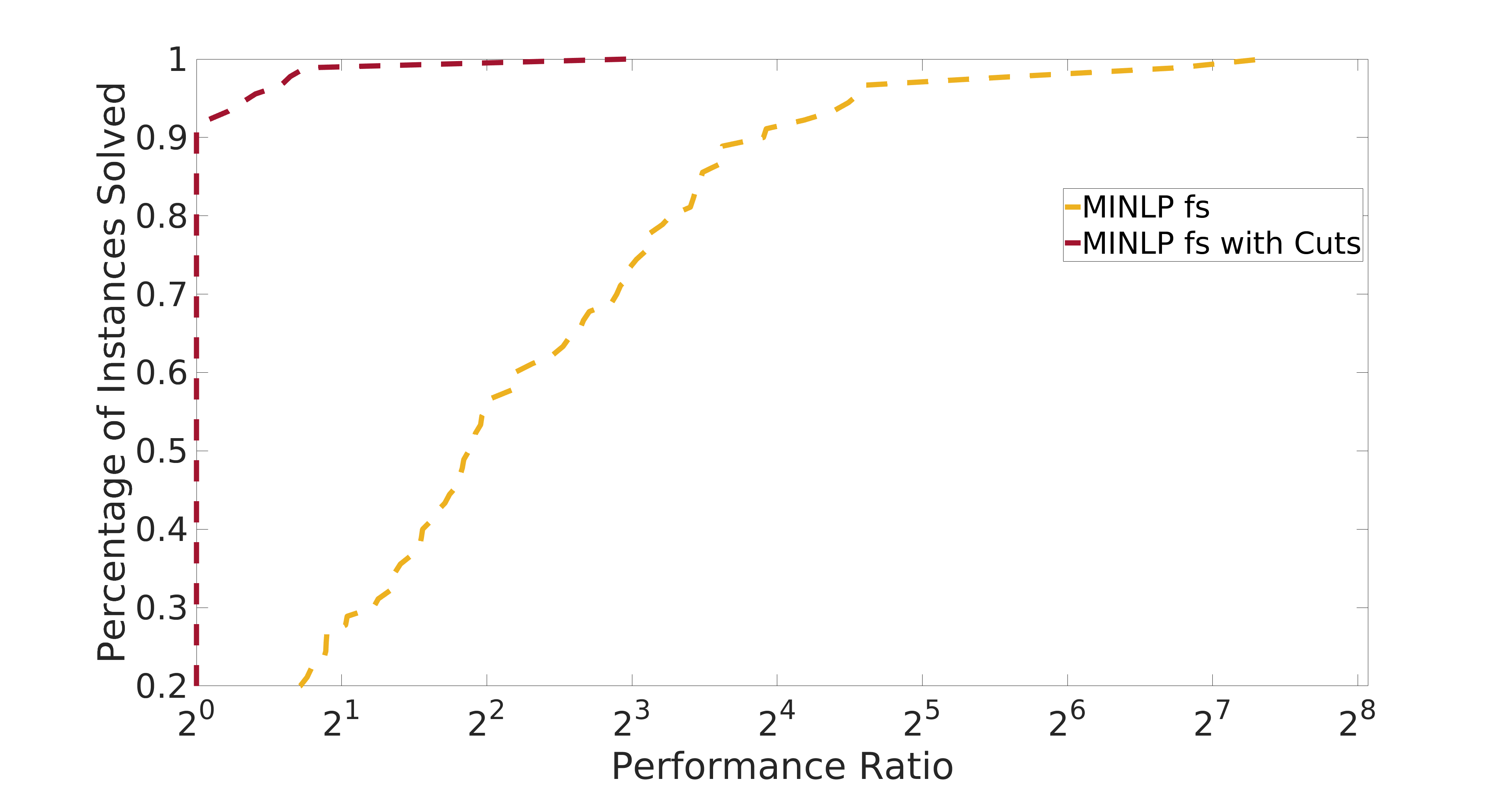}
       \caption{Cuts significantly accelerate MINLP with $\boldgreek{\phi}^{\textit{~fs}}$}
       \label{fig:perf:Cuts_MINLP_FS}
   \end{subfigure}    \begin{subfigure}[t]{0.49\textwidth}
        \centering
		\includegraphics[width=0.99\textwidth]{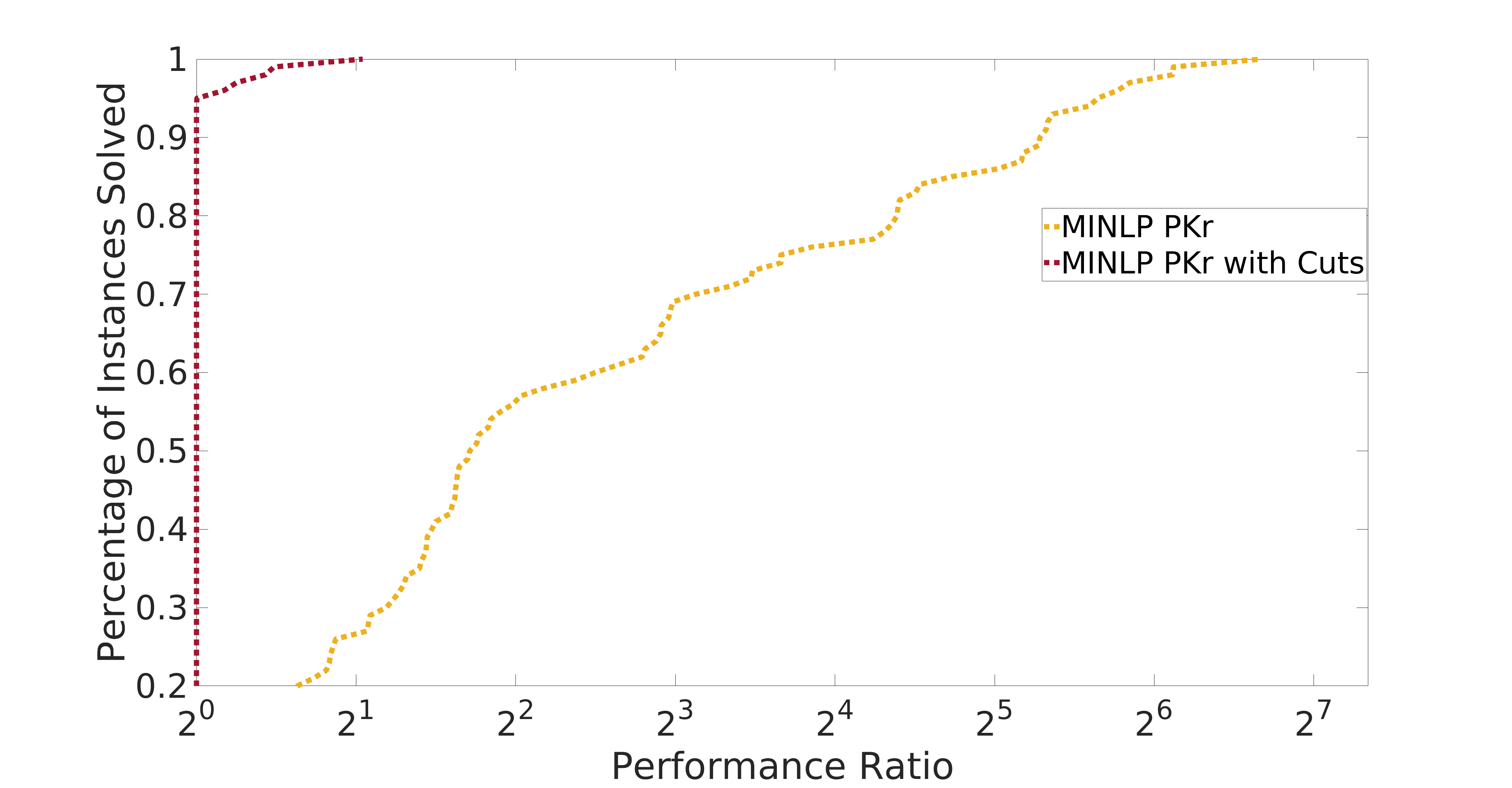}
        \caption{Cuts significantly accelerate MINLP with $\boldgreek{\phi}^{PKr}$}
        \label{fig:perf:Cuts_MINLP_Uni}
    \end{subfigure}
    \caption{Effects of cuts on the computational performance of the NLP and MINLP.}
    \label{fig:perf:CutsAllEffects}
\end{figure}

\subsubsection{Pressure Loss Model Comparisons}

\begin{figure}
    \centering
    \includegraphics[width=0.7\textwidth]{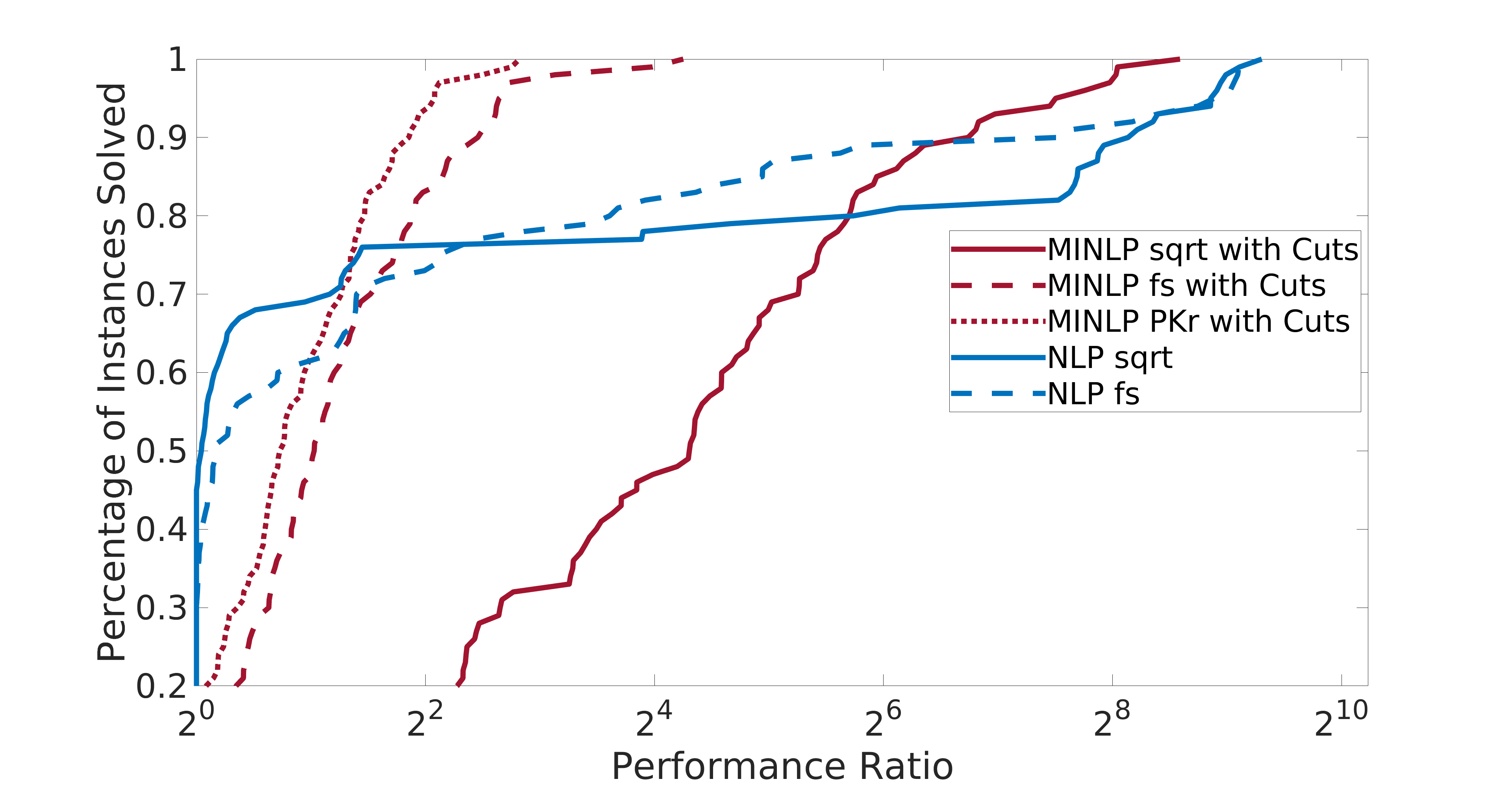}
    \caption{Performance Profile All Pressure Loss Models}
    \label{fig:Perf:MINLP_vs_MPCC_Best_ALL}
\end{figure}

Finally, we compare the fastest variants of all three pressure loss models (the NLP models without cuts and the MINLP models with cuts) in Figure \ref{fig:Perf:MINLP_vs_MPCC_Best_ALL} and Table \ref{tab:AllModels}. We omit the NLP with $\boldgreek{\phi}^{PKr}$ from this comparison because so few of those instances solved to optimality. From the comparison of the remaining five ``fastest'' models, we draw the following conclusions:
    \begin{itemize}
        \item As \cite{hante2019complementarity} discovered, the NLP model with $\boldgreek{\phi}^{sqrt}$ is bimodal: it is either very fast, or times out. The NLP with $\boldgreek{\phi}^{\textit{~fs}}$ had similar performance, except that it was slightly slower but was able to solve more instances (89 versus 81) within the 1 hour time limit. One of these two models had the fastest solve time in 75 out of the 100 instances (NLP with $\boldgreek{\phi}^{sqrt}$ won 45 times and NLP with $\boldgreek{\phi}^{\textit{~fs}}$ won 30 times).
        \item The MINLP with cuts had the most consistent performance. Regardless of the pressure loss model, all 100 instances of the MINLP with cuts solved to optimality within the 1 hour time limit.
        \item The new pressure loss formulations, $\boldgreek{\phi}^{\textit{~fs}}$ and $\boldgreek{\phi}^{PKr}$, significantly outshone $\boldgreek{\phi}^{sqrt}$ in the context of the MINLP with cuts. In addition to having stable and consistent performance, these two models were also fast, as indicated in Figure \ref{fig:Perf:MINLP_vs_MPCC_Best_ALL} and Table \ref{tab:AllModels}. In particular:
        \begin{itemize}
        \item The MINLP with $\boldgreek{\phi}^{\textit{~fs}}$ solved in under 50 seconds for 97 out of 100 instances. The slowest solve time for this model was 142.7 seconds. Head-to-head among all five models, this model was the fastest on 11 out of 100 gas network nominations.
        \item All instances of the MINLP with $\boldgreek{\phi}^{PKr}$ solved within 42 seconds, with 93 instances solving in under 30 seconds. This model was the fastest on 14 out of 100 nominations.
        \end{itemize}
\end{itemize}
In conclusion, the MINLP with $\boldgreek{\phi}^{\textit{~fs}}$ has the best performance among our studied models in terms of reliability, accuracy, and solve times utilizing the global solver BARON. While the MINLP with $\boldgreek{\phi}^{PKr}$ is slightly faster on average, the difference is likely not significant enough to outweigh the improved accuracy of $\boldgreek{\phi}^{\textit{~fs}}$. Furthermore, these results indicate that the MINLP, with added cuts and a modified pressure loss formulation, can be competitive with NLP.

\section{Conclusion} \label{sec:conclusion}

This paper considers various improvements to gas network optimization models that incorporate mixing and pressure loss. We consider improvements on two baseline nonlinear optimization models: a continuous nonlinear model and a discrete mixed-integer nonlinear model with binary variables. First we introduced three classes of cuts aimed at improving the performance of the discrete mixing formulation: (1) standard bilinear McCormick inequalities, (2) flow direction cuts using the binary variables, and (3) binary variable bounds on bilinear terms. Next, we derived two new smooth pressure loss formulations that take advantage of the directional flow variables required for the mixing formulations: (1) a ``flow-splitting'' formulation that is tuned to the most accurate pressure loss model (the HP-PC model) for both turbulent and laminar flows, and (2) a less accurate approximation that is equivalent to the existing PKr pressure loss approximation. We performed an extensive computational study and drew the following conclusions: 
\begin{itemize}
\item adding cuts significantly improved the discrete model (but hindered the continuous model);
\item with all pressure loss models, the MINLP with cuts was reliable  in that all 100 instances solved within the 1 hour time limit, whereas the NLP models (without cuts) were bimodal; they either solved very quickly or timed out;
\item in addition to consistent performance, the MINLP with cuts and either of the new pressure loss formulations was competitive with the NLP in terms of mean and geometric mean solve times on solved instances.
\end{itemize}

We remark that we performed our tests using the global nonlinear optimization solver, BARON \cite{sahinidis1996BARON}. In addition to BARON, there are other commercial nonlinear optimization solvers, such as ANTIGONE \cite{misener2014antigone}, SCIP \cite{bestuzheva2021scip}, Knitro \cite{waltz2004knitro}, CONOPT \cite{drud1994conopt}, SNOPT \cite{gill2005snopt}, Ipopt \cite{wachter2006implementation}, and MINOS \cite{murtagh1983minos}, that have recently been used in the gas network literature \cite{hante2019complementarity,schmidt2016high}. Among these, only BARON, ANTIGONE, and SCIP return globally optimal solutions for MINLPs. The remaining solvers are nonlinear solvers that find locally optimal solutions. These solvers may result in faster solution times than BARON on the tested NLPs. Because our interest was primarily on the MINLP and we had budget for only one commercial solver, we did not investigate the performance of these alternate solvers on the NLP. It would be valuable further research to compare the best MINLP model using BARON to the best NLP model using appropriate commercial nonlinear solvers.

\bibliographystyle{spmpsci}     
\bibliography{gasbib} 

\end{document}